\documentclass[10pt]{article}
\usepackage{lmodern}
\usepackage{amsmath}
\usepackage[T1]{fontenc}
\usepackage[utf8]{inputenc}
\usepackage{authblk}
\usepackage{amsfonts}
\usepackage{graphicx}
\usepackage{rotating}
\usepackage{amssymb}
\usepackage[english]{babel}
\usepackage{color}
\usepackage{amsthm}
\usepackage{graphicx}
\usepackage{mathrsfs}
\usepackage{makecell}
\usepackage{microtype}
\usepackage{enumerate}
\usepackage{mathscinet}
\usepackage{array}
\usepackage{multirow}
\usepackage{booktabs}
\usepackage{enumerate}
\usepackage[cal=boondoxo,bb=ams]{mathalfa}
\usepackage{hyperref}
\hypersetup{hidelinks}
\usepackage{titlesec}
\newtheorem{defn}{Definition}[section]

\newtheorem{theorem}{Theorem}[section]
\newtheorem{prop}{Proposition}[section]
\newtheorem{lemma}{Lemma}[section]

\newtheorem{remark}{Remark}[section]
\newtheorem{exam}{Example}[section]

\newcommand{\ml}{\mathcal}
\newcommand{\mb}{\mathbb}

\DeclareMathOperator{\non}{non}
\DeclareMathOperator{\lin}{lin}
\DeclareMathOperator{\intt}{int}
\DeclareMathOperator{\extt}{ext}
\DeclareMathOperator{\midd}{mid}

\title{The Cauchy problem for the Moore-Gibson-Thompson equation in the dissipative case}
\author[1]{Wenhui Chen\thanks{Corresponding author: Wenhui Chen (wenhui.chen.math@gmail.com)}}
\author[2]{Ryo Ikehata}
\affil[1]{School of Mathematical Sciences, Shanghai Jiao Tong University, 200240 Shanghai, China}
\affil[2]{Department of Mathematics, Division of Educational Sciences, Graduate School of Humanities and Social Sciences, Hiroshima University, 739-8524 Higashi-Hiroshima, Japan}

\date{}
\begin{document}

\maketitle
\begin{abstract}
In this paper, we study the Cauchy problem for the linear and semilinear Moore-Gibson-Thompson (MGT) equation in the dissipative case. Concerning the linear MGT model, by utilizing WKB analysis associated with Fourier analysis, we derive some $L^2$ estimates of solutions, which improve those in the previous research \cite{PellicerSaiHouari2017}. Furthermore, asymptotic profiles of the solution and an approximate relation in a framework of the weighted $L^1$ space are derived. Next, with the aid of the classical energy method and Hardy's inequality, we get singular limit results for an energy and the solution itself. Concerning the semilinear MGT model, basing on the obtained sharp $L^2$ estimates and constructing time-weighted Sobolev spaces, we investigate global (in time) existence of Sobolev solutions with different regularities. Finally, under a sign assumption on initial data, nonexistence of global (in time) weak solutions is proved by applying a test function method.\\
	
	\noindent\textbf{Keywords:} Moore-Gibson-Thompson equation, third-order hyperbolic equation, Fourier analysis, asymptotic profiles, singular limit, global existence.\\
	
	\noindent\textbf{AMS Classification (2010)} Primary:  35L30,  35L76; 35B40; Secondary: 74D05, 34K26, 35A01, 35B44.
\end{abstract}
\fontsize{12}{15}
\selectfont

\section{Introduction} 
In the last two decades, researches of the Moore-Gibson-Thompson (MGT) equation, which is linearized by a model for the wave propagation in viscous thermally relaxing fluids and is widely applied in medical as well as industrial uses of high-intensity ultrasound e.g. lithotripsy, thermotherapy or ultrasound cleaning, have caught a lot of attention. The MGT model is considered through the third-order (in time) strictly hyperbolic partial differential equation as follows: 
\begin{equation}\label{General.MGT.Equation}
\tau u_{ttt}+ u_{tt}-c^2\Delta u-b\Delta u_t=0,
\end{equation} 
where the scalar unknown $u=u(t,x)\in\mb{R}$ denotes an acoustic velocity. The MGT model \eqref{General.MGT.Equation} exhibits a variety of dynamical behaviors for solutions, which heavily depend on the physical parameters in the equation. To be specific, concerning the model \eqref{General.MGT.Equation}, $c$ stands for the speed of sound and $\tau$ denotes the thermal relaxation in the view of the physical context of acoustic waves. Moreover, the parameter $b=\beta c^2$ concerns the diffusivity of the sound carrying $\tau\in (0,\beta]$.

Actually, one may distinguish behaviors of solutions to the model \eqref{General.MGT.Equation} according to the dissipative case when $\tau\in(0,\beta)$ and the conservative case when $\tau=\beta$. Precisely, in the case of bounded domains for the linear MGT model, there exists a transition from the case $\tau\in(0,\beta)$ with an energy being exponentially stable to the limit case $\tau=\beta$ with an energy being conserved. Concerning some studies for the linear or nonlinear MGT equations, we refer interested readers to the related works  \cite{MooreGibson1960,Thompson1972,Kala-Tiry-1997,Gorain2010,KaltenbacherLasieckaMarchand2011,KaltenbacherLasiecka2012,MarchandMcDevittTriggiani2012,Jordan2014,LasieckaWang2015,PellicerSola2015,CaixetaLasieckaDominos2016,DellOroLasieckaPata2016,LasieckaWang2016,DellOroPata2017,Lasiecka2017,PellicerSaiHouari2017,BucciLasi,AlvesCaixetaSilvaRodrigues2018,DellOroLasieckaPata2019,BucciPandolfi,Racke-Said2019,PellicerSaid2019,Quintanilla2019,ChenPalmieri201901,ChenPalmieri201902,BucciEller2020,Nikoic-Said2020,Nikoic-Said202002} and references therein.

It is well-known that to study qualitative properties of solutions to the linear problem is not only significant for us to understand some underlying physical phenomena, it is also the crucial point for proving existence results of solutions to its corresponding nonlinear models. Let us come to the Cauchy problem for the linear MGT equation which has been firstly studied by the recent paper \cite{PellicerSaiHouari2017}. By reducing the third-order (in time) equation to the first-order (in time) coupled system, the authors of \cite{PellicerSaiHouari2017} employed energy methods in the Fourier space combined with suitable Lyapunov functionals to derive some energy estimates, and eigenvalues expansions to investigate some estimates for the solution itself. However, the obtained estimates for solutions in \cite{PellicerSaiHouari2017} seem not sharp, especially, in some low-dimensional cases. In this paper, we will improve their results and derive some optimal estimates. What's more, in the view of the limit case $\tau=0$, the linear MGT equation formally turns out to be the viscoelastic damped wave equation. For this reason, one may conjecture that there exist some relations between them. We will answer this conjecture from two points of view which are  singular limits and approximate relation in the sense of diffusion phenomena, respectively.

Our first aim in this paper is to investigate qualitative properties of solutions to the following linear MGT equation in the dissipative case:
\begin{align}\label{Linear_MGT_Dissipative}
\begin{cases}
\tau u_{ttt}+u_{tt}-\Delta u-\beta\Delta u_t=0,&x\in\mb{R}^n,\ t>0,\\
u(0,x)=0,\ u_t(0,x)=0,\ u_{tt}(0,x)=u_2(x),&x\in\mb{R}^n,
\end{cases}
\end{align}
where $\tau\in(0,\beta)$ and $n\geqslant 1$. Without loss of generality, we set the speed of the sound by $c^2=1$ in the last equation. To be specific, in Section \ref{Section_Linear_MGT} by preparing representation of solutions in the Fourier space and using asymptotic expansions of eigenvalues as well as WKB analysis, we deduce some $L^2$ estimates of solutions to the Cauchy problem \eqref{Linear_MGT_Dissipative} for initial data taken from $L^2$ space with or without additional $L^m$ regularity carrying $m\in[1,2)$. By a different treatment of some singularities, our results of $L^2$ estimates improve those in \cite{PellicerSaiHouari2017}, especially, the estimates of solutions in one and two spatial dimensions. Moreover, the regular assumption on initial data is relaxed. Later in Section \ref{Section_Asymptotic_Profile} we obtain asymptotic profiles of the solution to the Cauchy problem \eqref{Linear_MGT_Dissipative} in a framework of weighted $L^1$ data, where we provide sharp estimates for lower bounds and upper bounds of the solution itself in the $L^2$ norm. Namely, in the consideration of $L^2$ data with additional weighted $L^1$ regularity, the derived estimates are optimal for any $n\geqslant 1$. In Subsection \ref{Subsec_Approx_Relation}, in the frame of $L^2$ space, we describe an approximate relation (strongly related to diffusion phenomena) between the linear MGT equation and the linear viscoelastic damped wave equation, where gained decay rates are obtained for one- and two-dimensional cases. Next, in Section \ref{Sec_Singu_Limit} we consider the singular limit problem, in which we find the solution of the linear MGT equation converges to the solution of the linear viscoelastic damped wave equation as the thermal relaxation tending to $0$, i.e. $\tau\to 0^+$. Particularly, under different assumptions for initial data, we observe different rates of such tendency with respect to $\tau$. 

Our next purpose is to consider the Cauchy problem for the semilinear MGT equation in the dissipative case with the nonlinearity of power type, namely,
\begin{align}\label{Semi_Linear_MGT_Dissipative}
\begin{cases}
\tau u_{ttt}+u_{tt}-\Delta u-\beta\Delta u_t=|u|^p,&x\in\mb{R}^n,\ t>0,\\
u(0,x)=0,\ u_t(0,x)=0,\ u_{tt}(0,x)=u_2(x),&x\in\mb{R}^n,
\end{cases}
\end{align}
where $\tau\in(0,\beta)$, $n\geqslant 1$ and $p>1$. Recently, the blow-up results of the Cauchy problem for the semilinear MGT equation in the conservative case, i.e. the limit case $\tau=\beta$, with the nonlinearity of power type $|u|^p$ in \cite{ChenPalmieri201901}, or of derivative type $|u_t|^p$ in \cite{ChenPalmieri201902} have been obtained by applying iteration methods with suitable slicing procedure for unbounded multipliers. These works interpret the semilinear MGT equation in the conservative case as the semilinear wave equation with power source nonlinearities. Nevertheless, this statement does not hold anymore for the MGT equation in the dissipative case due to the damping effect that we derived in the corresponding linear problem. For this reason, it seems interesting to study existence as well as nonexistence of global (in time) solutions to the semilinear MGT models in the dissipative case.

Let us now turn to the Cauchy problem \eqref{Semi_Linear_MGT_Dissipative}. To the best of authors' knowledge, not only global (in time) existence but also blow-up results for \eqref{Semi_Linear_MGT_Dissipative} are still open. We will answer these questions in the present paper. By making use of the improved $L^2-L^2$ estimates with an additional $L^1$ regularity and employing Banach's fixed point theory, we prove global (in time) existence of small data Sobolev solutions to the Cauchy problem \eqref{Semi_Linear_MGT_Dissipative} in Section \ref{Section_Global_Existence}. Particularly, we analyze the interplay effect between dimension $n$, regularity $s$ and power $p$ on the existence of global (in time) Sobolev solution such that
\begin{align*}
u\in\ml{C}([0,\infty),H^s(\mb{R}^n)),
\end{align*}
with some positive parameters $s$. Soon afterward  in Section \ref{Section_Blow-up}, we apply a test function method to prove nonexistence of global (in time) weak solutions to the semilinear Cauchy problem \eqref{Semi_Linear_MGT_Dissipative} if the power $p$ fulfills some conditions. We should underline that the result in the one-dimensional case is optimal due to the blow-up result holding for any $1<p<\infty$.

Lastly, throughout Sections \ref{Section_Linear_MGT}, \ref{Section_Asymptotic_Profile}, \ref{Section_Global_Existence} and \ref{Section_Blow-up}, we will consider the MGT equations with vanishing first and second data. Indeed, non-vanishing third data will exert some dominant influences on the total estimates and existence results of solutions. We expect that one may derive the corresponding results with non-vanishing data by following the same approaches as we did later without any additional difficulties. Clearly, additional regularities for initial data would be necessary.

\medskip

\noindent\textbf{Notation: } We give some notations to be used in this paper. Later, $c$ and $C$ denote some positive constants, which may be changed from
line to line. We denote that $f\lesssim g$ if there exists a positive constant $C$ such that $f\leqslant Cg$ and, analogously, for $f\gtrsim g$. We denote $\lceil r\rceil:=\min\{y\in\mb{N}: 0<r\leqslant y\}$ as the positive ceiling function. $B_R$ stands for the ball around the origin with radius $R$ in $\mathbb{R}^n$. Moreover, $\dot{H}^s_q(\mb{R}^n)$ with $s\geqslant0$ and $1\leqslant q<\infty$, denote the Riesz potential spaces based on the Lebesgue spaces $L^q(\mb{R}^n)$. Finally, $|D|^s$ with $s\geqslant0$ stands for the pseudo-differential operator with the symbol $|\xi|^s$.

\section{Estimates of solutions to the linear MGT equation in the dissipative case}\label{Section_Linear_MGT}
\subsection{Pointwise estimates in the Fourier space}
At first, we apply the partial Fourier transform with respect to spatial variables to the Cauchy problem \eqref{Linear_MGT_Dissipative}. Then, it yields the following initial value problem for the third-order $|\xi|$-dependent ordinary differential equation:
\begin{align}\label{Linear_MGT_Dissipative_Fourier_Trans}
\begin{cases}
\tau\hat{u}_{ttt}+\hat{u}_{tt}+\beta|\xi|^2\hat{u}_t+|\xi|^2\hat{u}=0,&\xi\in\mb{R}^n,\ t>0,\\
\hat{u}(0,\xi)=0,\ \hat{u}_t(0,\xi)=0,\ \hat{u}_{tt}(0,\xi)=\hat{u}_2(\xi),&\xi\in\mb{R}^n,
\end{cases}
\end{align}
whose solution can be given by
\begin{align}\label{Representation_Fourier_u}
\hat{u}(t,\xi)=\widehat{K}(t,\xi)\hat{u}_2(\xi):=\left(\sum\limits_{j=1,2,3}\frac{\exp(\lambda_j(|\xi|)t)}{\prod_{k=1,2,3,\ k\neq j}\left(\lambda_j(|\xi|)-\lambda_k(|\xi|)\right)}\right)\hat{u}_2(\xi),
\end{align}
where $\lambda_j=\lambda_j(|\xi|)$ with $j=1,2,3,$ are three pairwise distinct roots to the cubic equation
\begin{align}\label{Cubic_Equation}
\tau\lambda^3+\lambda^2+\beta|\xi|^2\lambda+|\xi|^2=0.
\end{align} 
Here, the case for multiple roots can be regarded as a zero measure set with respect to $|\xi|$, and precisely, the  discriminant of \eqref{Cubic_Equation} is zero, that is
\begin{align*}
\triangle_{\mathrm{Cub}}=|\xi|^2\left(-4\beta^3\tau|\xi|^4+\left(18\beta\tau+\beta^2-27\tau^2\right)|\xi|^2-4\right)=0,
\end{align*}
if and only if
\begin{align}\label{Zero_Meas}
|\xi|^2=0\ \ \mbox{or}\ \ |\xi|^2=\frac{18\beta\tau+\beta^2-27\tau^2\pm\sqrt{(18\beta\tau+\beta^2-27\tau^2)^2-64\beta^3\tau}}{8\beta^3\tau}.
\end{align}
Under these preparations, we just need to discuss the case when the cubic equation \eqref{Cubic_Equation} does not have any roots of double multiply. Estimates of solutions in a zero measure set \eqref{Zero_Meas} do not give any influence on total estimates. Indeed, the pointwise estimates of solutions in the zero measure set were shown in \cite{PellicerSaiHouari2017}.
\begin{remark}\label{Rem_Finite_Prop_Speed}
	The principal symbol of the equation in \eqref{Linear_MGT_Dissipative} is given by $\tau\eta^3-\beta\eta|\zeta|^2$. Thus, the characteristic equation $\eta(\tau\eta^2-\beta|\zeta|^2)=0$ has pairwise distinct real roots $\eta=0$, $\eta=\sqrt{\beta/\tau}|\zeta|$ and $\eta=-\sqrt{\beta/\tau}|\zeta|$. In other words, the linear MGT equation in the dissipative case is strictly hyperbolic. Therefore, it is clear that the Cauchy problem \eqref{Linear_MGT_Dissipative} is well-posedness, e.g. there exists a unique Sobolev solution $u\in\ml{C}([0,\infty),H^s(\mb{R}^n))$ for $s\in[0,2]$ if $u_2\in H^{s-2}(\mb{R}^n)\subset L^2(\mb{R}^n)$. Furthermore, the theory in the strictly hyperbolic equation (see, for example, Section 3.4 in \cite{Taylor_Book}) shows that finite propagation speed property holds.
\end{remark}

Before deriving some $L^2$ estimates of solutions in the next subsection, we will prepare pointwise estimates of solutions in the Fourier space by investigating asymptotic behaviors of the kernel function $\widehat{K}(t,\xi)$. It is well-known that the explicit formula of the cubic equation \eqref{Cubic_Equation} can be uniquely given by Cardano's formula. Nevertheless, this would be a complex way to analyze behaviors of the kernel. To overcome the difficulty, we will employ asymptotic expansions of eigenvalues in small and large frequency zones, and demonstrate an exponential stability of solutions in bounded frequency zone. We define these zones in Fourier space by
\begin{align*}
Z_{\intt}(\varepsilon)&:=\left\{\xi\in\mb{R}^n:|\xi|<\varepsilon\ll1\right\},\\
Z_{\midd}(\varepsilon,N)&:=\left\{\xi\in\mb{R}^n:\varepsilon\leqslant |\xi|\leqslant N\right\},\\
Z_{\extt}(N)&:=\left\{\xi\in\mb{R}^n:|\xi|> N\gg1\right\}.
\end{align*}
Let us set the cut-off functions $\chi_{\intt}(\xi),\chi_{\midd}(\xi),\chi_{\extt}(\xi)\in \mathcal{C}^{\infty}(\mb{R}^n)$ owning their supports in $Z_{\intt}(\varepsilon)$, $Z_{\midd}(\varepsilon/2,2N)$ and $Z_{\extt}(N)$, respectively. Furthermore, they fulfill $\chi_{\midd}(\xi)=1-\chi_{\intt}(\xi)-\chi_{\extt}(\xi)$ for all $\xi \in \mb{R}^n$.

\begin{prop}\label{Prop_Pointwise_Estimate}
	Let $\tau\in(0,\beta)$. Then, the solution $\hat{u}=\hat{u}(t,\xi)$ to the initial value problem \eqref{Linear_MGT_Dissipative_Fourier_Trans} fulfills the following estimates:
	\begin{align}
	\chi_{\intt}(\xi)|\hat{u}(t,\xi)|&\lesssim\chi_{\intt}(\xi)\left(\left( |\cos(|\xi|t)|+\frac{|\sin(|\xi|t)|}{|\xi|}\right)\mathrm{e}^{-\frac{\beta-\tau}{2}|\xi|^2t}+\mathrm{e}^{-\frac{1}{\tau}t}\right)|\hat{u}_2(\xi)|,\label{Est_Small_Fouier}\\
	\chi_{\midd}(\xi)|\hat{u}(t,\xi)|&\lesssim\chi_{\midd}(\xi)\mathrm{e}^{-ct}|\hat{u}_2(\xi)|,\label{Est_Middle_Fouier}\\
	\chi_{\extt}(\xi)|\hat{u}(t,\xi)|&\lesssim\chi_{\extt}(\xi)\frac{1}{|\xi|^2}\mathrm{e}^{-\min\left\{\frac{\beta-\tau}{2\beta\tau},\frac{1}{\beta}\right\}t}|\hat{u}_2(\xi)|,\label{Est_Large_Fouier}
	\end{align}
	for some constants $c>0$.
\end{prop}

\begin{proof}
	Let us begin with estimating the solution $\hat{u}(t,\xi)$ for small frequencies. Motivated by the recent research \cite{PellicerSaiHouari2017}, we deduce that the eigenvalues $\lambda_j(|\xi|)$ with $j=1,2,3,$ have the asymptotic expansions for $|\xi|\to0$ such that
	\begin{align}\label{Asym_Expan_Small}
	\lambda_j(|\xi|)=
	\lambda_j^{(0)}+\lambda_j^{(1)}|\xi|+\lambda_j^{(2)}|\xi|^2+\cdots,
	\end{align}
	where the coefficients $\lambda_j^{(k)}\in\mb{C}$ for all $k\in\mb{N}_0$. What we need now is the dominant part of pairwise distinct eigenvalues. So, by plugging \eqref{Asym_Expan_Small} into \eqref{Cubic_Equation} and processing lengthy but straightforward computations, until different characteristic roots appear,
	the eigenvalues behave asymptotically for $|\xi|\to0$ as
	\begin{align*}
	\lambda_{1,2}(|\xi|)&=\pm i|\xi|-\frac{\beta-\tau}{2}|\xi|^2+\ml{O}(|\xi|^3),\\
	\lambda_3(|\xi|)&=-\frac{1}{\tau}+(\beta-\tau)|\xi|^2+\ml{O}(|\xi|^3).
	\end{align*}
	Let us denote a $|\xi|$-dependent function by
	\begin{align}\label{Defn_T0}
	\ml{T}_0(|\xi|):=\frac{1}{\tau}-\frac{3}{2}(\beta-\tau)|\xi|^2=\ml{O}(1)\ \ \mbox{as}\ \ |\xi|\to 0.
	\end{align}
	According to the representation of the kernel given in \eqref{Representation_Fourier_u}, the Fourier transform of the kernel localized in small frequency zone can be estimated by
	\begin{align*}
	\chi_{\intt}(\xi)|\widehat{K}(t,\xi)|&\lesssim \chi_{\intt}(\xi)\frac{\mathrm{e}^{-\frac{\beta-\tau}{2}|\xi|^2t}}{\ml{T}_0^2(|\xi|)+|\xi|^2}\left(|\cos(|\xi|t)|+\frac{|\sin(|\xi|t)|}{|\xi|}\ml{T}_0(|\xi|)\right)+\chi_{\intt}(\xi)\frac{\mathrm{e}^{-\frac{1}{\tau}t+(\beta-\tau)|\xi|^2t}}{\ml{T}_0^2(|\xi|)+|\xi|^2}\\
	&\lesssim\chi_{\intt}(\xi)\left(\left(|\cos(|\xi|t)|+\frac{|\sin(|\xi|t)|}{|\xi|}\right)\mathrm{e}^{-\frac{\beta-\tau}{2}|\xi|^2t}+\mathrm{e}^{-\frac{1}{\tau}t+(\beta-\tau)|\xi|^2t}\right),
	\end{align*}
	which immediately implies the desired estimate \eqref{Est_Small_Fouier}.
	
	Next, let us turn to the case for large frequencies. The eigenvalues to the cubic equation \eqref{Cubic_Equation} for $|\xi|\to\infty$ have the asymptotic expansions such that
	\begin{align}\label{Asym_Expan_Large}
	\lambda_j(|\xi|)
	=\bar{\lambda}_j^{(0)}|\xi|^2+\bar{\lambda}_j^{(1)}|\xi|+\bar{\lambda}_j^{(2)}+\bar{\lambda}_j^{(3)}|\xi|^{-1}+\cdots,
	\end{align}
	where the coefficients $\bar{\lambda}_j^{(k)}\in\mb{C}$ for all $k\in\mb{N}_0$. By substituting \eqref{Asym_Expan_Large} into \eqref{Cubic_Equation}, 
	it yields that the eigenvalues have asymptotic behaviors for $|\xi|\to\infty$ as follows:
	\begin{align*}
	\lambda_1(|\xi|)&=-\frac{1}{\beta}+\ml{O}(|\xi|^{-1}),\\
	\lambda_{2,3}(|\xi|)&=\pm i\frac{\sqrt{\beta}}{\sqrt{\tau}}|\xi|-\frac{\beta-\tau}{2\beta\tau}+\ml{O}(|\xi|^{-1}).
	\end{align*}
	Hence, the next chain inequalities hold:
	\begin{align*}
	\chi_{\extt}(\xi)|\widehat{K}(t,\xi)|&\lesssim
	\chi_{\extt}(\xi)\left(\frac{\mathrm{e}^{-\frac{1}{\beta}t}}{\left(\frac{\beta-3\tau}{2\beta\tau}\right)^2+\frac{\beta}{\tau}|\xi|^2}+\left(\frac{\left|\sin\left(\frac{\sqrt{\beta}}{\sqrt{\tau}}|\xi|t\right)\right|}{|\xi|\left(\left(\frac{\beta-3\tau}{2\beta\tau}\right)^2+\frac{\beta}{\tau}|\xi|^2\right)}+\frac{\left|\cos\left(\frac{\sqrt{\beta}}{\sqrt{\tau}}|\xi|t\right)\right|}{\left(\frac{\beta-3\tau}{2\beta\tau}\right)^2+\frac{\beta}{\tau}|\xi|^2}\right)\mathrm{e}^{-\frac{\beta-\tau}{2\beta\tau}t}\,\right)\\
	&\lesssim\chi_{\extt}(\xi)\frac{1}{|\xi|^2}\left(\mathrm{e}^{-\frac{\beta-\tau}{2\beta\tau}t}+\mathrm{e}^{-\frac{1}{\beta}t}\right).
	\end{align*}
	The previous estimate combined with the formula of solution in \eqref{Representation_Fourier_u} proves our desired assertion \eqref{Est_Large_Fouier}.
	
	Finally, let us prove an exponential decay estimate of solutions localized in bounded frequency zone. With the aim of deriving an exponential stability of solutions, we now follow the idea of Subsection 2.3 in \cite{Chen.JMAA}. Let us assume that there exists an eigenvalue $\lambda=id$ with $d\in\mb{R}\backslash\{0\}$. In other words, according to \eqref{Cubic_Equation}, the non-zero real number $d$ should fulfill the equalities
	\begin{align*}
	id\left(\tau d^2-\beta|\xi|^2\right)=0\ \ \mbox{and}\ \ d^2-|\xi|^2=0.
	\end{align*}
	Due to the settings that $d\neq0$ and $\tau\in(0,\beta)$, it immediately finds a contradiction. Namely, there does not exists any pure imaginary roots to the cubic equation \eqref{Cubic_Equation} for $\xi\in Z_{\midd}(\varepsilon,N)$. Viewing the expansions of eigenvalues, we know $\mathrm{Re}\,\lambda_j(|\xi|)<0$ for any $j=1,2,3$ as $\xi\in Z_{\intt}(\varepsilon)\cup Z_{\extt}(N)$. Therefore, by applying the compactness of bounded frequency zone $Z_{\midd}(\varepsilon,N)$ and the continuity of the eigenvalues, the derivation of the exponential decay estimates \eqref{Est_Middle_Fouier} and the proof of this proposition are complete.
\end{proof}

\subsection{$L^2$ estimates of solutions}
Basing on the pointwise estimates shown in Proposition \ref{Prop_Pointwise_Estimate}, we next investigate $L^2-L^2$ estimates with or without additional $L^m$ regularity with $m\in[1,2)$, respectively. These estimates will play an essential  role in the forthcoming part to consider global (in time) existence of solutions to the semilinear MGT model. 

\begin{theorem}\label{Thm_L2-L2_Est}
	Let $\tau\in(0,\beta)$. Then, the solution $u=u(t,x)$ to the Cauchy problem \eqref{Linear_MGT_Dissipative} fulfills the following estimates:
	\begin{align*}
	\|\,|D|^su(t,\cdot)\|_{L^2(\mb{R}^n)}\lesssim
	\begin{cases}
	(1+t)^{1-\frac{s}{2}}\|u_2\|_{H^{\max\{s-2,0\}}(\mb{R}^n)}&\mbox{if}\ \ s\in[0,1),\\
	(1+t)^{\frac{1}{2}-\frac{s}{2}}\|u_2\|_{H^{\max\{s-2,0\}}(\mb{R}^n)}&\mbox{if}\ \ s\in[1,\infty),
	\end{cases} 
	\end{align*}
	for any $t\geqslant0$.
\end{theorem}
\begin{proof}
	By applying Proposition \ref{Prop_Pointwise_Estimate} and the Parseval equality, we arrive at
	\begin{align}
	\|\,|D|^su(t,\cdot)\|_{L^2(\mb{R}^n)}&\lesssim\left\|\chi_{\intt}(\xi)|\xi|^s\left(\left(|\cos(|\xi|t)|+\frac{|\sin(|\xi|t)|}{|\xi|}\right)\mathrm{e}^{-\frac{\beta-\tau}{2}|\xi|^2t}+\mathrm{e}^{-\frac{1}{\tau}t}\right)\right\|_{L^{\infty}(\mb{R}^n)}\|u_2\|_{L^2(\mb{R}^n)}\notag\\
	&\quad+\mathrm{e}^{-ct}\|u_2\|_{L^2(\mb{R}^n)}+\mathrm{e}^{-ct}\left\|\chi_{\extt}(\xi)|\xi|^{s-2}\hat{u}_2(\xi)\right\|_{L^2(\mb{R}^n)},\label{Supp_01}
	\end{align}
	with the aid of  the norm inequality $\|\cdot\|_{L^2(\mb{R}^n)}\leqslant\|\cdot\|_{L^{\infty}(\mb{R}^n)}\|\cdot\|_{L^2(\mb{R}^n)}$. 
	
	Let us estimate the first $L^{\infty}$ norm on the right-hand side of \eqref{Supp_01}. Obviously, by using $|\cos(|\xi|t)|\leqslant 1$, then for any $t\geqslant 0$ we get
	\begin{align*}
	\left\|\chi_{\intt}(\xi)|\xi|^s|\cos(|\xi|t)|\mathrm{e}^{-\frac{\beta-\tau}{2}|\xi|^2t}\right\|_{L^{\infty}(\mb{R}^n)}\lesssim \left\|\chi_{\intt}(\xi)|\xi|^s\mathrm{e}^{-\frac{\beta-\tau}{2}|\xi|^2t}\right\|_{L^{\infty}(\mb{R}^n)}\lesssim (1+t)^{-\frac{s}{2}}.
	\end{align*}
	We will divide the remaindering estimate into two parts.\\ Concerning the case for small time, i.e. $t\in[0,1]$, one may directly obtain
	\begin{align*}
	\left\|\chi_{\intt}(\xi)|\xi|^{s-1}|\sin(|\xi|t)|\mathrm{e}^{-\frac{\beta-\tau}{2}|\xi|^2t}\right\|_{L^{\infty}(\mb{R}^n)}\lesssim
	t\left\|\chi_{\intt}(\xi)|\xi|^s\frac{|\sin(|\xi|t)|}{|\xi|t}\mathrm{e}^{-\frac{\beta-\tau}{2}|\xi|^2t}\right\|_{L^{\infty}(\mb{R}^n)}\lesssim1,
	\end{align*}
	which immediately shows bounded estimates for small time.\\
	For another, concerning the case for large time, i.e. $t\in(1,\infty)$, one applies $|\sin(|\xi|t)|\leqslant 1$ to have
	\begin{align*}
	\left\|\chi_{\intt}(\xi)|\xi|^{s-1}|\sin(|\xi|t)|\mathrm{e}^{-\frac{\beta-\tau}{2}|\xi|^2t}\right\|_{L^{\infty}(\mb{R}^n)}&\lesssim t^{-\frac{s-1}{2}}\left\|\chi_{\intt}(\xi)\left(|\xi|^2t\right)^{\frac{s-1}{2}}\mathrm{e}^{-\frac{\beta-\tau}{2}|\xi|^2t}\right\|_{L^{\infty}(\mb{R}^n)}\lesssim t^{\frac{1}{2}-\frac{s}{2}}
	\end{align*}
	for $s\in[1,\infty)$. In the case $s\in[0,1)$, we do by another way that
	\begin{align*}
	\left\|\chi_{\intt}(\xi)|\xi|^{s-1}|\sin(|\xi|t)|\mathrm{e}^{-\frac{\beta-\tau}{2}|\xi|^2t}\right\|_{L^{\infty}(\mb{R}^n)}\lesssim
	t^{1-\frac{s}{2}}\left\|\chi_{\intt}(\xi)\left(|\xi|^2t\right)^{\frac{s}{2}}\frac{|\sin(|\xi|t)|}{|\xi|t}\mathrm{e}^{-\frac{\beta-\tau}{2}|\xi|^2t}\right\|_{L^{\infty}(\mb{R}^n)}\lesssim t^{1-\frac{s}{2}}.
	\end{align*}
	Thus, it completes that
	\begin{align*}
	\left\|\chi_{\intt}(\xi)|\xi|^{s-1}|\sin(|\xi|t)|\mathrm{e}^{-\frac{\beta-\tau}{2}|\xi|^2t}\right\|_{L^{\infty}(\mb{R}^n)}\lesssim\begin{cases}
	(1+t)^{1-\frac{s}{2}}&\mbox{if}\ \  s\in[0,1),\\
	(1+t)^{\frac{1}{2}-\frac{s}{2}}&\mbox{if}\ \ s\in[1,\infty),
	\end{cases} 
	\end{align*}
	for any $t\geqslant0$. Particularly, decay estimates hold for any $s\in(1,\infty)$.
	
	On the other hand, we know
	\begin{align*}
	\left\|\chi_{\extt}(\xi)|\xi|^{s-2}\hat{u}_2(\xi)\right\|_{L^2(\mb{R}^n)}\lesssim\|u_2\|_{H^{\max\{s-2,0\}}(\mb{R}^n)},
	\end{align*}
	where we used $\chi_{\extt}(\xi)|\xi|^{s-2}\lesssim 1$ if $s\in[0,2]$ and $\chi_{\extt}(\xi)|\xi|^{s-2}\lesssim (1+|\xi|^2)^{(s-2)/2}$ if $s\in(2,\infty)$.
	
	Summarizing the derived estimates, the proof is now complete.
\end{proof}

\begin{theorem}\label{Thm_Lm-L2_Est}
	Let $\tau\in(0,\beta)$. Then, the solution $u=u(t,x)$ to the Cauchy problem \eqref{Linear_MGT_Dissipative} fulfills the following estimates:
	\begin{align*}
	\|\,|D|^su(t,\cdot)\|_{L^2(\mb{R}^n)}\lesssim
	\begin{cases}
	\ml{F}(t)\|u_2\|_{H^{\max\{s-2,0\}}(\mb{R}^n)\cap L^m(\mb{R}^n)}&\mbox{if}\ \ 2sm+(2-m)n< 2+m,\\
	(1+t)^{\frac{1}{2}-\frac{s}{2}-\frac{n(2-m)}{4m}}\|u_2\|_{H^{\max\{s-2,0\}}(\mb{R}^n)\cap L^m(\mb{R}^n)}&\mbox{if}\ \ 2sm+(2-m)n\geqslant 2+m,
	\end{cases} 
	\end{align*}
	for any $t\geqslant0$, where $s\in[0,\infty)$ and $m\in[1,2)$. In the above case $2sm+(2-m)n< 2+m$, the time-dependent coefficient  is denoted by
	\begin{align*}
	\ml{F}(t):=\begin{cases}
	(1+t)^{1-s-\frac{n(2-m)}{2m}}&\mbox{if}\ \ 2sm+(2-m)n<2m,\\
	(1+t)^{\frac{1}{2}-\frac{s}{2}-\frac{n(2-m)}{4m}}(\ln (\mathrm{e}+t))^{\frac{2-m}{2m}}&\mbox{if}\ \ 2sm+(2-m)n=2m,\\
	(1+t)^{\frac{1}{2}-\frac{s}{2}-\frac{n(2-m)}{4m}+\frac{2+m-2sm-(2-m)n}{2(2+m)}}&\mbox{if}\ \ 2sm+(2-m)n>2m.
	\end{cases}
	\end{align*}
\end{theorem}
\begin{remark}\label{Rem_Improv_Said}
	Let us consider the special case $m=1$. The estimates stated in Theorem \ref{Thm_Lm-L2_Est} improve those results of Theorem 5.1 and Theorem 5.3 in \cite{PellicerSaiHouari2017}. For example, concerning the estimate of the solution itself, i.e. $s=0$, according to Theorem \ref{Thm_Lm-L2_Est}, we arrive at
	\begin{align*}
	\|u(t,\cdot)\|_{L^2(\mb{R}^n)}\lesssim\begin{cases}
	(1+t)^{\frac{1}{2}}\|u_2\|_{L^2(\mb{R})\cap L^1(\mb{R})}&\mbox{if}\ \ n=1,\\
	(\ln(\mathrm{e}+t))^{\frac{1}{2}}\|u_2\|_{L^2(\mb{R}^2)\cap L^1(\mb{R}^2)}&\mbox{if}\ \ n=2,\\
	(1+t)^{\frac{1}{2}-\frac{n}{4}}\|u_2\|_{L^2(\mb{R}^n)\cap L^1(\mb{R}^n)}&\mbox{if}\ \ n\geqslant 3,
	\end{cases}
	\end{align*}
	where the derived estimates in the low-dimensional cases $n=1$ and $n=2$ are sharper than those in \cite{PellicerSaiHouari2017}. General speaking, we replace the restriction $s+n\geqslant 3$ in Theorem 5.3 shown in \cite{PellicerSaiHouari2017}  by $2s+n\geqslant 3$, which allows us to get shaper estimates in a larger admissible range of dimensions, e.g. $n=2$ with $s=1/2$. For another, the requirement of the regularity for initial data is relaxed from $H^s$ to $H^{\max\{s-2,0\}}$.
\end{remark}
\begin{proof}
	We may start by discussing the case for small frequencies. Employing H\"older's inequality and the Hausdorff-Young inequality, one has
	\begin{align*}
	&\|\chi_{\intt}(D)|D|^su(t,\cdot)\|_{L^2(\mb{R}^n)}\\
	&\lesssim\left\|\chi_{\intt}(\xi)|\xi|^s\left(\left(|\cos(|\xi|t)|+\frac{|\sin(|\xi|t)|}{|\xi|}\right)\mathrm{e}^{-\frac{\beta-\tau}{2}|\xi|^2t}+\mathrm{e}^{-\frac{1}{\tau}t}\right)\right\|_{L^{\frac{2m}{2-m}}(\mb{R}^n)}\|u_2\|_{L^m(\mb{R}^n)}\\
	&\lesssim\left\|\chi_{\intt}(\xi)\left(|\xi|^s|\cos(|\xi|t)|+|\xi|^{s-1}|\sin(|\xi|t)|\right)\mathrm{e}^{-\frac{\beta-\tau}{2}|\xi|^2t}\right\|_{L^{\frac{2m}{2-m}}(\mb{R}^n)}\|u_2\|_{L^m(\mb{R}^n)}+\mathrm{e}^{-\frac{1}{\tau}t}\|u_2\|_{L^m(\mb{R}^n)}\\
	&\lesssim\left(\int_0^{\varepsilon}r^{\frac{2(s-1)m}{2-m}+n-1}|\sin(rt)|^{\frac{2m}{2-m}}\mathrm{e}^{-\frac{(\beta-\tau)m}{2-m}r^2t}\mathrm{d}r\right)^{\frac{2-m}{2m}}\|u_2\|_{L^m(\mb{R}^n)}+(1+t)^{-\frac{s}{2}-\frac{n(2-m)}{4m}}\|u_2\|_{L^m(\mb{R}^n)},
	\end{align*}
	where we applied
	\begin{align*}
	\left\|\chi_{\intt}(\xi)|\xi|^s|\cos(|\xi|t)|\mathrm{e}^{-\frac{\beta-\tau}{2}|\xi|^2t}\right\|_{L^{\frac{2m}{2-m}}(\mb{R}^n)}&\lesssim\left(\int_0^{\varepsilon}r^{\frac{2sm}{2-m}+n-1}|\cos(rt)|^{\frac{2m}{2-m}}\mathrm{e}^{-\frac{(\beta-\tau)m}{2-m}r^2t}\mathrm{d}r\right)^{\frac{2-m}{2m}}\\
	&\lesssim(1+t)^{-\frac{s}{2}-\frac{n(2-m)}{4m}}.
	\end{align*}
	Let us now estimate the term including the sine function which is denoted by
	\begin{align*}
	\ml{G}(t):=\left(\int_0^{\varepsilon}r^{\frac{2(s-1)m}{2-m}+n-1}|\sin(rt)|^{\frac{2m}{2-m}}\mathrm{e}^{-\frac{(\beta-\tau)m}{2-m}r^2t}\mathrm{d}r\right)^{\frac{2-m}{2m}}.
	\end{align*}
	Due to the interplay between the diffusive part from $\exp\left(-\frac{(\beta-\tau)m}{2-m}r^2t\right)$ and the oscillating part from $|\sin(rt)|/r$, one should analyze a delicate equilibrium as well as the singularity as $r\to0^+$ in the case for negative power of $r$. This treatment is the difference from those in \cite{PellicerSaiHouari2017}. For one thing, as usual approach by considering $t\in[0,1]$, we find
	\begin{align*}
	\ml{G}(t)=t \left(\int_0^{\varepsilon}r^{\frac{2ms}{2-m}+n-1}\left(\frac{|\sin(rt)|}{rt}\right)^{\frac{2m}{2-m}}\mathrm{e}^{-\frac{(\beta-\tau)m}{2-m}r^2t}\mathrm{d}r\right)^{\frac{2-m}{2m}}\lesssim 1,
	\end{align*}
	where we used $2ms/(2-m)+n-1\geqslant0$. For another, we consider $t\in(1,\infty)$ to derive 
	\begin{align*}
	\ml{G}(t)&= t^{-\frac{2(s-1)m+n(2-m)}{4m}}\left(\int_0^{\varepsilon}(r^2t)^{\frac{2(s-1)m+(n-1)(2-m)}{2(2-m)}}\mathrm{e}^{-\frac{(\beta-\tau)m}{2-m}r^2t}\mathrm{d}(r^2t)^{\frac{1}{2}}\right)^{\frac{2-m}{2m}}\lesssim t^{\frac{1}{2}-\frac{s}{2}-\frac{n(2-m)}{4m}},
	\end{align*}
	where we restricted $2sm+(2-m)n\geqslant 2+m$ to guarantee the nonnegativity of the power for $r^2t$ in the integral term, otherwise, a singularity will come as $r\to 0^+$. 
	\\ Let us use another approach to get the result when $2sm+(2-m)n< 2+m$ for $t\in(1,\infty)$. Setting a new variable $\omega=rt^{\frac{1}{2}}$, it holds that
	\begin{align}\label{G(t)bound}
	\ml{G}(t)\lesssim t^{-\frac{s-1}{2}-\frac{n(2-m)}{4m}}(\ml{I}(t))^{\frac{2-m}{2m}},
	\end{align}
	where the time-dependent function on the right-hand side is defined by
	\begin{align*}
	\ml{I}(t):=\ml{I}^{(1)}(t)+\ml{I}^{(2)}(t):=\left(\int_0^{t^{-1/\alpha}}+\int_{t^{-1/\alpha}}^{\infty}\right)\omega^{\frac{2sm+(2-m)n-(2+m)}{2-m}}|\sin(t^{1/2}\omega)|^{\frac{2m}{2-m}}\mathrm{e}^{-\frac{(\beta-\tau)m}{2-m}\omega^2}\mathrm{d}\omega.
	\end{align*}
	Here, we used WKB analysis to separate the integral over $(0,\infty)$ to $(0,t^{-1/\alpha})$ and $[t^{-1/\alpha},\infty)$ carrying a suitable positive constant $\alpha$ to be determined later. The choice of the parameter $\alpha$ is helpful for us to understand sharper estimates.\\
	To estimate $\ml{I}^{(1)}(t)$,  by the boundedness of $|\sin(y)/y|$, we obtain
	\begin{align}\label{Est_Ike_01}
	\ml{I}^{(1)}(t)&\lesssim t^{\frac{m}{2-m}}\int_0^{t^{-1/\alpha}}\left|\frac{\sin(t^{1/2}\omega)}{t^{1/2}\omega}\right|^{\frac{2m}{2-m}}\omega^{\frac{2sm+(2-m)n}{2-m}-1}\mathrm{e}^{-\frac{(\beta-\tau)m}{2-m}\omega^2}\mathrm{d}\omega\notag\\
	&\lesssim t^{\frac{m\alpha-2sm-(2-m)n+2-m}{(2-m)\alpha}}\int_{0}^{t^{-1/\alpha}}\mathrm{d}\omega\lesssim t^{\frac{m\alpha-2sm-(2-m)n}{(2-m)\alpha}},
	\end{align}
	where we observed $2sm+(2-m)n\geqslant (2-m)$ for any $s\in[0,\infty)$ and $m\in[1,2)$.\\
	To investigate the estimate for $\ml{I}^{(2)}(t)$, we divide the discussion into three cases. If $2sm+(2-m)n<2m$, then we may directly apply integration by parts to find
	\begin{align}\label{Est_Ike_02}
	\ml{I}^{(2)}(t)&\lesssim\int_{t^{-1/\alpha}}^{\infty}\omega^{\frac{2sm+(2-m)n-(2+m)}{2-m}}\mathrm{e}^{-\frac{(\beta-\tau)m}{2-m}\omega^2}\mathrm{d}\omega\notag\\
	&\lesssim\frac{2-m}{2sm+(2-m)n-2m}\left(\omega^{\frac{2sm+(2-m)n-2m}{2-m}}\mathrm{e}^{-\frac{(\beta-\tau)m}{2-m}\omega^2}\right)\Big|_{\omega=t^{-1/\alpha}}^{\omega=\infty}\notag\\
	&\quad+\frac{2(\beta-\tau)m}{2sm+(2-m)n-2m}\int_{t^{-1/\alpha}}^{\infty}\omega^{\frac{2sm+(2-m)n-2m}{2-m}+1}\mathrm{e}^{-\frac{(\beta-\tau)m}{2-m}\omega^2}\mathrm{d}\omega\notag\\
	&\lesssim t^{\frac{2m-2sm-(2-m)n}{(2-m)\alpha}}\mathrm{e}^{-\frac{(\beta-\tau)m}{2-m}t^{-2/\alpha}}-\int_{t^{-1/\alpha}}^{\infty}\omega^{\frac{2sm+(2-m)n+2-3m}{2-m}}\mathrm{e}^{-\frac{(\beta-\tau)m}{2-m}\omega^2}\mathrm{d}\omega\notag\\
	&\lesssim t^{\frac{2m-2sm-(2-m)n}{(2-m)\alpha}}.
	\end{align}
	By considering \eqref{Est_Ike_01} and \eqref{Est_Ike_02}, in the case $2sm+(2-m)n<2m$, we may obtain the sharp estimates
	\begin{align*}
	\ml{I}(t)\lesssim t^{\frac{m\alpha-2sm-(2-m)n}{(2-m)\alpha}}+t^{\frac{2m-2sm-(2-m)n}{(2-m)\alpha}}\lesssim t^{\frac{2m-2sm-(2-m)n}{2(2-m)}},
	\end{align*}
	providing that $m\alpha-2sm-(2-m)n=2m-2sm-(2-m)n$ if and only if $\alpha=2$.\\
	Let us turn to the case $2sm+(2-m)n=2m$. Therefore, by the similar procedure to the above, we estimate
	\begin{align}\label{Est_Ike_03}
	\ml{I}^{(2)}(t)&\lesssim\int_{t^{-1/\alpha}}^{\infty}\omega^{-1}\mathrm{e}^{-\frac{(\beta-\tau)m}{2-m}\omega^2}\mathrm{d}\omega\notag\\
	&\lesssim\left((\ln\omega)\mathrm{e}^{-\frac{(\beta-\tau)m}{2-m}\omega^2}\right)\Big|_{\omega=t^{-1/\alpha}}^{\omega=\infty}+\frac{2(\beta-\tau)m}{2-m}\int_{t^{-1/\alpha}}^{\infty}\omega|\ln\omega|\mathrm{e}^{-\frac{(\beta-\tau)m}{2-m}\omega^2}\mathrm{d}\omega\notag\\
	&\lesssim\frac{1}{\alpha}(\ln t)\mathrm{e}^{-\frac{(\beta-\tau)m}{2-m}t^{-2/\alpha}}+\int_{0}^{\infty}\omega|\ln\omega|\mathrm{e}^{-\frac{(\beta-\tau)m}{2-m}\omega^2}\mathrm{d}\omega\lesssim \ln t.
	\end{align}
	Then, by choosing $\alpha=2$ again, it follows from \eqref{Est_Ike_01} and \eqref{Est_Ike_03} that
	\begin{align*}
	\ml{I}(t)\lesssim t^{\frac{m\alpha-2sm-(2-m)n}{(2-m)\alpha}}+\ln t\lesssim \ln t,
	\end{align*}
	when $2sm+(2-m)n=2m$.\\
	In the remaindering case $2sm+(2-m)n>2m$, we found that
	\begin{align}\label{Est_Ike_04}
	\ml{I}^{(2)}(t)&\lesssim\int_{t^{-1/\alpha}}^{\infty}\omega^{\frac{2sm+(2-m)n-(2+m)}{2-m}}\mathrm{e}^{-\frac{(\beta-\tau)m}{2-m}\omega^2}\mathrm{d}\omega\notag\\
	&\lesssim t^{\frac{2+m-2sm-(2-m)n}{(2-m)\alpha}}\int_{0}^{\infty}\mathrm{e}^{-\frac{(\beta-\tau)m}{2-m}\omega^2}\mathrm{d}\omega\lesssim t^{\frac{2+m-2sm-(2-m)n}{(2-m)\alpha}}.
	\end{align}
	For the moment, we would like to remark that since
	\begin{align*}
	\int_0^{\infty}\mathrm{e}^{-\frac{(\beta-\tau)m}{2-m}\omega^2}\mathrm{d}\omega=\sqrt{\frac{\pi(2-m)}{2(\beta-\tau)m}},
	\end{align*}
	the restriction on the dissipative case, i.e. $\tau\in(0,\beta)$, acts a pivotal part in the way that
	\begin{align*}
	\lim\limits_{\tau\to\beta^-}\int_0^{\infty}\mathrm{e}^{-\frac{(\beta-\tau)m}{2-m}\omega^2}\mathrm{d}\omega=\infty,
	\end{align*}
	which somehow shows the limit case $\tau=\beta$ having singularities.\\
	Combining \eqref{Est_Ike_01} and \eqref{Est_Ike_04}, it yields
	\begin{align*}
	\ml{I}(t)\lesssim t^{\frac{m\alpha-2sm-(2-m)n}{(2-m)\alpha}}+t^{\frac{2+m-2sm-(2-m)n}{(2-m)\alpha}}\lesssim t^{\frac{m(2+m-2sm-(2-m)n)}{(2+m)(2-m)}},
	\end{align*}
	where we chose $\alpha=(2+m)/m$ to guarantee the optimality of the last competition.\\
	All in all, from \eqref{G(t)bound}, for $t\in(1,\infty)$ we assert that
	\begin{align*}
	\ml{G}(t)\lesssim\begin{cases}
	t^{1-s-\frac{n(2-m)}{2m}}&\mbox{if}\ \ 2sm+(2-m)n<2m,\\
	t^{\frac{1}{2}-\frac{s}{2}-\frac{n(2-m)}{4m}}(\ln t)^{\frac{2-m}{2m}}&\mbox{if}\ \ 2sm+(2-m)n=2m,\\
	t^{\frac{1}{2}-\frac{s}{2}-\frac{n(2-m)}{4m}+\frac{2+m-2sm-(2-m)n}{2(2+m)}}&\mbox{if}\ \ 2sm+(2-m)n>2m,
	\end{cases}
	\end{align*}
	in the case  $2sm+(2-m)n<2+m$.
	
	The estimates of solutions for bounded frequencies and large frequencies are exactly the  same as those in Theorem \ref{Thm_L2-L2_Est}. Thus, the summary of the derived estimates completes the proof.
\end{proof}

\section{Asymptotic profiles in a framework of weighted $L^1$ space}\label{Section_Asymptotic_Profile}
\subsection{Optimal estimates with weighted $L^1$ data}
In this subsection, we will derive asymptotic profiles for the linear MGT equation in the dissipative case in a framework of $L^{1,1}$ space, where
\begin{align*}
L^{1,1}(\mb{R}^n):=\left\{f\in L^1(\mb{R}^n):\ \|f\|_{L^{1,1}(\mb{R}^n)}:=\int_{\mb{R}^n}(1+|x|)|f(x)|\mathrm{d}x<\infty\right\}.
\end{align*}
As a preparation, we now define a time-dependent function 
\begin{align*}
\ml{D}_n(t):=\begin{cases}
t^{\frac{1}{2}}&\mbox{if}\ \ n=1,\\
(\ln t)^{\frac{1}{2}}&\mbox{if}\ \ n=2,\\
t^{-\frac{n-2}{4}}&\mbox{if} \ \ n\geqslant 3.
\end{cases}
\end{align*}
In order to derive asymptotic profiles of solutions, we will estimate upper bounds and lower bounds of the solution itself with $u_2\in L^2(\mb{R}^n)\cap L^{1,1}(\mb{R}^n)$.  Before processing these estimates, let us introduce the notation for the integral of $f(x)$ by $P_{f}:=\int_{\mb{R}^n}f(x)\mathrm{d}x$, and 
recall Lemma 2.1 from \cite{Ikehata2004}.
\begin{lemma}\label{Lemma_Ikehata_MMAS2004}
	Let us assume $f\in L^{1,1}(\mb{R}^n)$. Then, the following estimate holds:
	\begin{align*}
	|\hat{f}(\xi)|\leqslant C_1|\xi|\,\|f\|_{L^{1,1}(\mb{R}^n)}+|P_f|,
	\end{align*}
	with a positive constant $C_1>0$.
\end{lemma}
Moreover, due to the support condition for $\chi_{\intt}(\xi)$, by minor modifications of some derived lemmas in \cite{Ikehata2014,IkehataOnodera2017}, one may show the validity of Lemma \ref{Lemma_Ikehata_JDE_02}. Or one may use the inequality
\begin{align*}
\|\chi_{\intt}(\xi)\hat{f}(t,|\xi|)\|_{L^2(\mb{R}^n)}\gtrsim \|\hat{f}(t,|\xi|)\|_{L^2(\mb{R}^n)}-\left(\|\chi_{\midd}(\xi)\hat{f}(t,|\xi|)\|_{L^2(\mb{R}^n)}+\|\chi_{\extt}(\xi)\hat{f}(t,|\xi|)\|_{L^2(\mb{R}^n)}\right)
\end{align*}
for $t\gg1$, where $\hat{f}(t,|\xi|)=|\sin(|\xi|t)|\mathrm{e}^{-c|\xi|^2t}/|\xi|$ or $\hat{f}(t,|\xi|)=|\cos(|\xi|t)|\mathrm{e}^{-c|\xi|^2t}$ and estimates
\begin{align*}
\|\chi_{\midd}(\xi)\hat{f}(t,|\xi|)\|_{L^2(\mb{R}^n)}+\|\chi_{\extt}(\xi)\hat{f}(t,|\xi|)\|_{L^2(\mb{R}^n)}\lesssim \mathrm{e}^{-c_0t},
\end{align*}
with a suitable constant $c_0>0$, to prove the next lemma. 
\begin{lemma}\label{Lemma_Ikehata_JDE_02}
	Let $n\geqslant 1$. The following estimates hold:
	\begin{align*}
	\ml{D}_n(t)\lesssim\left\|\chi_{\intt}(\xi)\frac{|\sin(|\xi|t)|}{|\xi|}\mathrm{e}^{-c|\xi|^2t}\right\|_{L^2(\mb{R}^n)}&\lesssim\ml{D}_n(t),\\
	t^{-\frac{n}{4}}\lesssim\left\|\chi_{\intt}(\xi)|\cos(|\xi|t)|\mathrm{e}^{-c|\xi|^2t}\right\|_{L^2(\mb{R}^n)} &\lesssim t^{-\frac{n}{4}},
	\end{align*}
	with $c>0$, for $t\gg1$.
\end{lemma}

Let us state our result on asymptotic profiles of the solution. Particularly, in one and two-dimensional cases, we can easily observe the glow-up properties of the solution $u(t,\cdot)$ in the $L^2$ norm for the linear MGT equation in the dissipative case with initial data belonging to $L^2\cap L^{1,1}$. 
\begin{theorem}\label{THM_Asymptoti_Profiles}
	Let $\tau\in(0,\beta)$. Let us assume $u_2\in L^2(\mb{R}^n)\cap L^{1,1}(\mb{R}^n)$ and $|P_{u_2}|\neq0$. Then, the solution $u=u(t,x)$ to the Cauchy problem \eqref{Linear_MGT_Dissipative} fulfills the following estimates:
	\begin{align*}
	\ml{D}_n(t)|P_{u_2}|\lesssim\|u(t,\cdot)\|_{L^2(\mb{R}^n)}&\lesssim \ml{D}_n(t)\|u_2\|_{L^2(\mb{R}^n)\cap L^{1,1}(\mb{R}^n)}
	\end{align*}
	for any $t\gg1$.
\end{theorem}
\begin{remark}
	According to Theorem \ref{THM_Asymptoti_Profiles} and concerning $t\gg1$, we may observe that the decay rate for the estimates of $\|u(t,\cdot)\|_{L^2(\mb{R}^n)}$ from the above and the below are the same for any $n\geqslant 1$. Moreover, $u_2\in L^{1,1}(\mb{R}^n)$ implies $|P_{u_2}|<\infty$ for $n\geqslant 1$.  Namely, the decay estimates stated in Theorem \ref{THM_Asymptoti_Profiles} are optimal for all spatial dimensions in a framework of weighted $L^1$ space.
\end{remark}
\begin{proof}
	Initially, let us estimate upper bounds of solutions by modifying the estimate for small frequencies. The philosophy of derivative is essentially the same as those in Theorem \ref{Thm_Lm-L2_Est}. By applying Lemma \ref{Lemma_Ikehata_MMAS2004} and Proposition \ref{Prop_Pointwise_Estimate}, we arrive at
	\begin{align*}
	\chi_{\mathrm{int}}(\xi)|\hat{u}(t,\xi)|&\lesssim\chi_{\mathrm{int}}(\xi)\left(\left(|\xi|\,|\cos(|\xi|t)|+|\sin(|\xi|t)|\right)\mathrm{e}^{-\frac{\beta-\tau}{2}|\xi|^2t}+|\xi|\mathrm{e}^{-\frac{1}{\tau}t}\right)\|u_2\|_{L^{1,1}(\mb{R}^n)}\\
	&\quad+\chi_{\mathrm{int}}(\xi)\left(\left(|\cos(|\xi|t)|+\frac{|\sin(|\xi|t)|}{|\xi|}\right)\mathrm{e}^{-\frac{\beta-\tau}{2}|\xi|^2t}+\mathrm{e}^{-\frac{1}{\tau}t}\right)|P_{u_2}|.
	\end{align*}
	Clearly, for the sake of the polar co-ordinate transform, we may deduce
	\begin{align*}
	&\left\|\chi_{\mathrm{int}}(\xi)\left(\left(|\xi|\,|\cos(|\xi|t)|+|\sin(|\xi|t)|\right)\mathrm{e}^{-\frac{\beta-\tau}{2}|\xi|^2t}+|\xi|\mathrm{e}^{-\frac{1}{\tau}t}\right)\right\|_{L^2(\mb{R}^n)}\\
	&\lesssim\left(\int_0^{\varepsilon}r^{n+1}|\cos(rt)|^2\mathrm{e}^{-(\beta-\tau)r^2t}\mathrm{d}r\right)^{\frac{1}{2}}+\left(\int_0^{\varepsilon}r^{n-1}|\sin(rt)|^2\mathrm{e}^{-(\beta-\tau)r^2t}\mathrm{d}r\right)^{\frac{1}{2}}+\mathrm{e}^{-\frac{1}{\tau}t}\\
	&\lesssim t^{-\frac{n+2}{4}}+ t^{-\frac{n}{4}}+\mathrm{e}^{-\frac{1}{\tau}t}\lesssim t^{-\frac{n}{4}}
	\end{align*}
	for $t\gg1$. Repeating the same procedure as those in Theorem \ref{Thm_Lm-L2_Est}, we have
	\begin{align}\label{Improve_L11=0}
	\|\chi_{\mathrm{int}}(D)u(t,\cdot)\|_{L^2(\mb{R}^n)}\lesssim t^{-\frac{n}{4}}\|u_2\|_{L^{1,1}(\mb{R}^n)}+\ml{D}_n(t)|P_{u_2}|
	\end{align}
	for $t\gg 1$. In the case for bounded and large frequencies, we just need to use the same estimates as those in Theorem \ref{Thm_L2-L2_Est}. Finally, by using the fact that $	|P_{u_2}|\leqslant \|u_2\|_{L^{1,1}(\mb{R}^n)}$, we are able to prove upper bound estimates for the solution itself in the $L^2$ norm.
	
	Let us now turn to lower bound estimates. According to the study in Section \ref{Section_Linear_MGT}, we may represent the solution for small frequencies by
	\begin{align*}
	\chi_{\mathrm{int}}(\xi)\hat{u}(t,\xi)&=\chi_{\mathrm{int}}(\xi)I(t,|\xi|)\hat{u}_2(\xi)\\
	&=\chi_{\mathrm{int}}(\xi)(I_1(t,|\xi|)+I_2(t,|\xi|)+I_3(t,|\xi|))\hat{u}_2(\xi),
	\end{align*}
	where for the sake of convenience in the proof, we denoted $I(t,|\xi|):=\widehat{K}(t,\xi)$ which was shown in \eqref{Representation_Fourier_u}, furthermore, we introduced
	\begin{align*}
	I_{1,2}(t,|\xi|)&:=\frac{\exp\left(\left(\pm i|\xi|-\frac{\beta-\tau}{2}|\xi|^2+\ml{O}(|\xi|^3)\right)t\right)}{\left(\pm 2i|\xi|+\ml{O}(|\xi|^3)\right)\left(\frac{1}{\tau}\pm i|\xi|-\frac{3}{2}(\beta-\tau)|\xi|^2+\ml{O}(|\xi|^3)\right)},\\
	I_3(t,|\xi|)&:=\frac{\exp\left(\left(-\frac{1}{\tau}+(\beta-\tau)|\xi|^2+\ml{O}(|\xi|^3)\right)t\right)}{
		\left(\frac{1}{\tau}+i|\xi|-\frac{3}{2}(\beta-\tau)|\xi|^2+\ml{O}(|\xi|^3)\right)\left(\frac{1}{\tau}-i|\xi|-\frac{3}{2}(\beta-\tau)|\xi|^2+\ml{O}(|\xi|^3)\right)}.
	\end{align*}
	By omitting the terms containing $\ml{O}(|\xi|^3)$, we may regard next three functions as the leading term of $I_j(t,|\xi|)$ for small frequencies:
	\begin{align*}
	J_{1,2}(t,|\xi|)&:=\frac{\exp\left(\left(\pm i|\xi|-\frac{\beta-\tau}{2}|\xi|^2\right)t\right)}{\pm 2i|\xi|\left(\frac{1}{\tau}\pm i|\xi|-\frac{3}{2}(\beta-\tau)|\xi|^2\right)},\\
	J_3(t,|\xi|)&:=\frac{\exp\left(\left(-\frac{1}{\tau}+(\beta-\tau)|\xi|^2\right)t\right)}{\left(\frac{1}{\tau}+i|\xi|-\frac{3}{2}(\beta-\tau)|\xi|^2\right)\left(\frac{1}{\tau}-i|\xi|-\frac{3}{2}(\beta-\tau)|\xi|^2\right)},
	\end{align*}
	respectively, whose sum can be shown by
	\begin{align}\label{Defn_J}
	J(t,|\xi|)&:=\sum\limits_{k=1,2,3}J_k(t,|\xi|)=\frac{\mathrm{e}^{-\frac{\beta-\tau}{2}|\xi|^2t}}{\ml{T}_0^2(|\xi|)+|\xi|^2}\left(\frac{\sin(|\xi|t)}{|\xi|}\ml{T}_0(|\xi|)+\mathrm{e}^{-\frac{1}{\tau}t+\frac{3}{2}(\beta-\tau)|\xi|^2t}-\cos(|\xi|t)\right).
	\end{align}
	where we recalled \eqref{Defn_T0}. Now, we should be carefully analyze that the error estimates between the leading term $J_j(t,|\xi|)$ the formulas $I_j(t,|\xi|)$ for $j=1,2,3$, individually. It proves the additional decay estimates. Concerning the case for $J_1(t,|\xi|)$, denoting
	\begin{align*}
	g_1(|\xi|):=\frac{1}{\tau}+i|\xi|-\frac{3}{2}(\beta-\tau)|\xi|^2=\ml{O}(1)\ \ \mbox{as}\ \ |\xi|\to 0,
	\end{align*}
	we may handle
	\begin{align*}
	\chi_{\mathrm{int}}(\xi)|I_1(t,|\xi|)-J_1(t,|\xi|)|&\lesssim\chi_{\mathrm{int}}(\xi)\mathrm{e}^{-\frac{\beta-\tau}{2}|\xi|^2t}\left|\frac{\mathrm{e}^{\ml{O}(|\xi|^3)t}}{2i|\xi|g_1(|\xi|)+\ml{O}(|\xi|^3)}-\frac{1}{2i|\xi|g_1(|\xi|)}\right|\\
	&\lesssim\chi_{\mathrm{int}}(\xi)\mathrm{e}^{-\frac{\beta-\tau}{2}|\xi|^2t}\left|\frac{2i|\xi|g_1(|\xi|)\left(\mathrm{e}^{\ml{O}(|\xi|^3)t}-1\right)+\ml{O}(|\xi|^3)}{4|\xi|^2(g_1(|\xi|))^2+\ml{O}(|\xi|^4)}\right|\\
	&\lesssim\chi_{\mathrm{int}}(\xi)\mathrm{e}^{-\frac{\beta-\tau}{2}|\xi|^2t}\frac{1}{|\xi|^2}\left(\ml{O}(|\xi|^4)t\int_0^1\mathrm{e}^{\ml{O}(|\xi|^3)ts}\mathrm{d}s-\ml{O}(|\xi|^3)\right)\\
	&\lesssim\chi_{\mathrm{int}}(\xi)\left(\ml{O}(|\xi|^2)t\mathrm{e}^{-c|\xi|^2t}+\ml{O}(|\xi|)\mathrm{e}^{-c|\xi|^2t}\right),
	\end{align*}
	since there exists a constant $c>0$ such that
	\begin{align*}
	\chi_{\intt}(\xi)\mathrm{e}^{-\frac{\beta-\tau}{2}|\xi|^2t}\int_0^1\mathrm{e}^{\ml{O}(|\xi|^3)ts}\mathrm{d}s\lesssim\chi_{\intt}(\xi)\mathrm{e}^{-\frac{\beta-\tau}{4}|\xi|^2t}\,\mathrm{e}^{-\frac{1}{4}((\beta-\tau)-\ml{O}(|\xi|))|\xi|^2t} \lesssim \chi_{\intt}(\xi)\mathrm{e}^{-c|\xi|^2t}.
	\end{align*}
	Next, by repeating the same way as the previous one, we get
	\begin{align*}
	\chi_{\mathrm{int}}(\xi)|I_2(t,|\xi|)-J_2(t,|\xi|)|\lesssim\chi_{\mathrm{int}}(\xi)\left(\ml{O}(|\xi|^2)t\mathrm{e}^{-c|\xi|^2t}+\ml{O}(|\xi|)\mathrm{e}^{-c|\xi|^2t}\right).
	\end{align*}
	Considering the last term, by defining
	\begin{align*}
	g_2(|\xi|):=\left(\frac{1}{\tau}-\frac{3}{2}(\beta-\tau)|\xi|^2\right)^2+|\xi|^2=\ml{O}(1)\ \ \mbox{as}\ \ |\xi|\to 0,
	\end{align*} one has
	\begin{align*}
	\chi_{\mathrm{int}}(\xi)|I_3(t,|\xi|)-J_3(t,|\xi|)|&\lesssim\chi_{\mathrm{int}}(\xi)\mathrm{e}^{-\frac{1}{\tau}t+(\beta-\tau)|\xi|^2t}\left|\frac{\mathrm{e}^{\ml{O}(|\xi|^3)t}}{g_2(|\xi|)+\ml{O}(|\xi|^3)}-\frac{1}{g_2(|\xi|)}\right|\\
	&\lesssim \chi_{\mathrm{int}}(\xi)\mathrm{e}^{-\frac{1}{\tau}t+(\beta-\tau)|\xi|^2t}\left|\frac{g_2(|\xi|)\ml{O}(|\xi|^3)t\int_0^1\mathrm{e}^{\ml{O}(|\xi|^3)ts}\mathrm{d}s-\ml{O}(|\xi|^3)}{(g_2(|\xi|))^2+\ml{O}(|\xi|^3)}\right|\\
	&\lesssim\chi_{\mathrm{int}}(\xi)\mathrm{e}^{-ct}\ml{O}(|\xi|^3)t\lesssim \chi_{\mathrm{int}}(\xi)\mathrm{e}^{-ct}
	\end{align*}
	for any $t\gg 1$, which still provides us an exponential decay.\\
	In conclusion, we can claim
	\begin{align}\label{Est_001}
	\chi_{\mathrm{int}}(\xi)|(I(t,|\xi|)-J(t,|\xi|))\hat{u}_2(\xi)|\lesssim \chi_{\mathrm{int}}(\xi)\left( \ml{O}(|\xi|^2)t+\ml{O}(|\xi|)\right)\mathrm{e}^{-c|\xi|^2t}|\hat{u}_2(\xi)|.
	\end{align}
	Let us decompose initial data by
	\begin{align*}
	\hat{u}_2(\xi)=P_{u_2}+A(\xi)-iB(\xi),
	\end{align*}
	where
	\begin{align*}
	A(\xi):=\int_{\mb{R}^n}u_2(x)(1-\cos(x\cdot\xi))\mathrm{d}x\ \ \mbox{and}\ \ B(\xi):=\int_{\mb{R}^n}u_2(x)\sin(x\cdot\xi)\mathrm{d}x. 
	\end{align*}
	In the view of Lemma 2.2 in \cite{Ikehata2014}, these $\xi$-dependent functions can be controlled by
	\begin{align*}
	|A(\xi)|+|B(\xi)|\lesssim |\xi|\,\|u_2\|_{L^{1,1}(\mb{R}^n)}.
	\end{align*}
	As a consequence, we may represent the solution in the Fourier space by
	\begin{align*}
	\hat{u}(t,\xi)=I(t,|\xi|)P_{u_2}+(A(\xi)-iB(\xi))I(t,|\xi|).
	\end{align*}
	From the derived estimate \eqref{Est_001}, it yields
	\begin{align}\label{Chen_00001}
	&\left\|\chi_{\mathrm{int}}(D)u(t,\cdot)-\chi_{\mathrm{int}}(D)\ml{F}_{\xi\to x}^{-1}(J(t,|\xi|))P_{u_2}\,\right\|_{L^2(\mb{R}^n)}\notag\\
	&\leqslant\left\|\chi_{\mathrm{int}}(\xi)(I(t,|\xi|)-J(t,|\xi|))\right\|_{L^2(\mb{R}^n)}|P_{u_2}|+\left\|\chi_{\mathrm{int}}(\xi)(A(\xi)-iB(\xi))I(t,|\xi|)\right\|_{L^2(\mb{R}^n)}\notag\\
	&\lesssim\left\|\chi_{\mathrm{int}}(\xi)\left(|\xi|^2t+|\xi|\right)\mathrm{e}^{-c|\xi|^2t}\right\|_{L^2(\mb{R}^n)}|P_{u_2}|+\|\chi_{\mathrm{int}}(\xi)|\xi|I(t,|\xi|)\|_{L^2(\mb{R}^n)}\|u_2\|_{L^{1,1}(\mb{R}^n)}\notag\\
	&\lesssim t^{-\frac{n}{4}}\|u_2\|_{L^{1,1}(\mb{R}^n)}
	\end{align}
	for $t\gg1$, where we used
	\begin{align*}
	\|\chi_{\mathrm{int}}(\xi)|\xi|I(t,|\xi|)\|_{L^2(\mb{R}^n)}\lesssim\left\|\chi_{\mathrm{int}}(\xi)(|\xi|+1)\mathrm{e}^{-\frac{\beta-\tau}{2}|\xi|^2t}\right\|_{L^2(\mb{R}^n)}\lesssim t^{-\frac{n}{4}}.
	\end{align*}
	Additionally, let us recall the function $J(t,|\xi|)$ in \eqref{Defn_J}. By employing Lemma \ref{Lemma_Ikehata_JDE_02} and the  Parseval equality, it is valid that
	\begin{align*}
	&\left\|\chi_{\mathrm{int}}(D)\ml{F}_{\xi\to x}^{-1}(J(t,|\xi|))\right\|_{L^2(\mb{R}^n)}\\
	&\gtrsim\left\|\chi_{\intt}(\xi)\left(\frac{\sin(|\xi|t)}{|\xi|}\ml{T}_0(|\xi|)+\mathrm{e}^{-\frac{1}{\tau}t+\frac{3}{2}(\beta-\tau)|\xi|^2t}-\cos(|\xi|t)\right)\mathrm{e}^{-\frac{\beta-\tau}{2}|\xi|^2t}\right\|_{L^2(\mb{R}^n)}\\
	&\gtrsim\left|\left\|\chi_{\intt}(\xi)\ml{H}(t,|\xi|)\mathrm{e}^{-\frac{\beta-\tau}{2}|\xi|^2t}\right\|_{L^2(\mb{R}^n)}-\left\|\chi_{\intt}(\xi)\cos(|\xi|t)\mathrm{e}^{-\frac{\beta-\tau}{2}|\xi|^2t}\right\|_{L^2(\mb{R}^n)}\right|\\
	&\gtrsim|
	\ml{D}_n(t)-t^{-\frac{n}{4}}|\gtrsim\ml{D}_n(t)
	\end{align*}
	for $t\gg1$, where we denoted
	\begin{align*}
	\ml{H}(t,|\xi|):=\frac{\sin(|\xi|t)}{|\xi|}\ml{T}_0(|\xi|)+\mathrm{e}^{-\frac{1}{\tau}t+\frac{3}{2}(\beta-\tau)|\xi|^2t},
	\end{align*}
	moreover, we used upper bound estimates as follows:
	\begin{align*}
	\left\|\chi_{\intt}(\xi)\ml{H}(t,|\xi|)\mathrm{e}^{-\frac{\beta-\tau}{2}|\xi|^2t}\right\|_{L^2(\mb{R}^n)}\lesssim \ml{D}_n(t)
	\end{align*}
	and estimates from the below such that
	\begin{align*}
	&\left\|\chi_{\intt}(\xi)\ml{H}(t,|\xi|)\mathrm{e}^{-\frac{\beta-\tau}{2}|\xi|^2t}\right\|_{L^2(\mb{R}^n)}\\
	&\gtrsim\left|\left\|\chi_{\intt}(\xi)\mathrm{e}^{-\frac{1}{\tau}t+(\beta-\tau)|\xi|^2t}\right\|_{L^2(\mb{R}^n)}-\left\|\chi_{\intt}(\xi)\frac{|\sin(|\xi|t)|}{|\xi|}\mathrm{e}^{-\frac{\beta-\tau}{2}|\xi|^2t}\right\|_{L^2(\mb{R}^n)}\right|\\
	&\gtrsim |\mathrm{e}^{-ct}-\ml{D}_n(t)|\gtrsim\ml{D}_n(t)
	\end{align*}
	for any $n\geqslant 1$ and $t\gg1$.\\
	Finally, by using the Minkowski inequality, we conclude
	\begin{align*}
	\|\chi_{\mathrm{int}}(D)u(t,\cdot)\|_{L^2(\mb{R}^n)}&\geqslant\left\|\chi_{\mathrm{int}}(D)\ml{F}_{\xi\to x}^{-1}(J(t,|\xi|))\right\|_{L^2(\mb{R}^n)}|P_{u_2}|\\
	&\quad-\left\|\chi_{\mathrm{int}}(D)u(t,\cdot)-\chi_{\mathrm{int}}(D)\ml{F}_{\xi\to x}^{-1}(J(t,|\xi|))P_{u_2}\,\right\|_{L^2(\mb{R}^n)}\\
	&\gtrsim \ml{D}_n(t)|P_{u_2}|-t^{-\frac{n}{4}}\|u_2\|_{L^{1,1}(\mb{R}^n)}
	\end{align*}
	for $t\gg1$. Actually, in the above by taking $t\gg1$, the time-dependent coefficients of $|P_{u_2}|$ play dominant influence for all $n\geqslant1$. Thus, with the help of the fact that 
	\begin{align*}
	\|u(t,\cdot)\|_{L^2(\mb{R}^n)}\geqslant \|\chi_{\mathrm{int}}(D)u(t,\cdot)\|_{L^2(\mb{R}^n)}\gtrsim
	\ml{D}_n(t)|P_{u_2}|,
	\end{align*}
	the proof is complete.
\end{proof}
\subsection{Approximate relation in one- and two-dimensional cases}\label{Subsec_Approx_Relation}
Our purpose in this part is to give an approximate relation between the linear MGT equation and the linear viscoelastic damped wave equation (or the strongly damped wave equation). This approximate relation is strongly related to the so-called diffusion phenomenon (see, for example, \cite{Nishihara2003}), which bridges a connection for the damped wave equation and the heat equation such that
\begin{align*}
\begin{cases}
\tau u^{\mathrm{dw}}_{tt}-\Delta u^{\mathrm{dw}}+u^{\mathrm{dw}}_t=0,&x\in\mb{R}^n,\ t>0,\\
u^{\mathrm{dw}}(0,x)=0,\ u^{\mathrm{dw}}_t(0,x)=u^{\mathrm{dw}}_1(x),&x\in\mb{R}^n,
\end{cases}
\ \
\begin{aligned}
\Longrightarrow\ \ \\
\tau= 0
\end{aligned}
\ \ 
\begin{cases}
-\Delta v^{\mathrm{h}}+v^{\mathrm{h}}_t=0,&x\in\mb{R}^n,\ t>0,\\
v^{\mathrm{h}}(0,x)=u^{\mathrm{dw}}_1(x),&x\in\mb{R}^n.
\end{cases}
\end{align*}
It is well-known that the decay rates of $u^{\mathrm{dw}}(t,\cdot)$ and $v^{\mathrm{h}}(t,\cdot)$ in the $L^2$ norm are the same. Furthermore, the decay estimates of the difference
\begin{align*}
\left\|u^{\mathrm{dw}}(t,\cdot)-v^{\mathrm{h}}(t,\cdot)\right\|_{L^2(\mb{R}^n)}
\end{align*} 
is faster than the decay estimates for each of them in the $L^2$ norm. The gained decay rate is $(1+t)^{-1}$. Namely, diffusion phenomena bridge the connection between second-order (in time) evolution equations and first-order (in time) evolution equations.

Before giving our result, let us recall some derived estimates of solutions to the following linear Cauchy problem:
\begin{align}\label{Linear_Visco_Damped_Wave}
\begin{cases}
\tilde{u}_{tt}-\Delta\tilde{u}-\beta\Delta\tilde{u}_t=0,&x\in\mb{R}^n,\ t>0,\\
\tilde{u}(0,x)=0,\ \tilde{u}_t(0,x)=\tilde{u}_1(x),&x\in\mb{R}^n,
\end{cases}
\end{align}
where $\beta>0$. The Cauchy problem for the viscoelastic damped wave equation has been deeply studied in \cite{Shibata2000,IkehataTodorovaYordanov2013,DabbiccoReissig2014,Ikehata2014,IkehataNatsume,IkehataOnodera2017,BarreraVolkmer2019,BarreraVolkmer2020} and references therein. Particularly, in the paper \cite{Ikehata2014}, the author proved estimates of solutions to \eqref{Linear_Visco_Damped_Wave} as follows:
\begin{align}\label{Visco_Est}
\|\tilde{u}(t,\cdot)\|_{L^2(\mb{R}^n)}\lesssim\ml{D}_n(t)\|\tilde{u}_1\|_{L^2(\mb{R}^n)\cap L^{1,1}(\mb{R}^n)}
\end{align}
for $t\gg1$, providing that $|P_{\tilde{u}_1}|\neq0$. Concerning the Cauchy problem, we found that the estimates for the linear MGT equation \eqref{Linear_MGT_Dissipative} in Theorem \ref{THM_Asymptoti_Profiles}, and for the viscoelastic damped wave equation \eqref{Linear_Visco_Damped_Wave} in \eqref{Visco_Est}, are exactly the same. Therefore, we conjecture that behaviors of solutions for the linear MGT equation  are similar to those for the linear viscoelastic damped wave equation, especially the decay property. Furthermore, it becomes interesting to derive the approximate relation between them with suitable initial data, and to find a gained decay rate.

From the previous study, we know the decay rates of $(L^2\cap L^{1,1})-L^2$ estimates are determined by the behavior of the eigenvalues for small frequencies only. In the case for bounded and large frequencies, the behaviors of the eigenvalues together with the suitable regularity for initial data show immediately some exponential decays. For this reason, the next approximate relation is explained by the behavior of solutions localized in small frequency zone, which is the most interesting one.

\begin{theorem}\label{THM_APPROX}
	Let $\tau\in(0,\beta)$. Let us assume $u_2\in L^{1,1}(\mb{R}^n)$ and $|P_{u_2}|\neq0$. Then, the solution $u=u(t,x)$ to the Cauchy problem \eqref{Linear_MGT_Dissipative} and the solution $\tilde{u}=\tilde{u}(t,x)$ to the Cauchy problem \eqref{Linear_Visco_Damped_Wave} with $\tilde{u}_1(x)=u_2(x)$ fulfill the following estimates:
	\begin{align*}
	\left\|\chi_{\intt}(D)\left(u(t,\cdot)-\tau\tilde{u}(t,\cdot)\right)\right\|_{L^2(\mb{R}^n)}\lesssim t^{\frac{1}{2}-\frac{n}{4}}\|u_2\|_{L^{1,1}(\mb{R}^n)}
	\end{align*}
	for any $n\geqslant 1$ and $t\gg1$.
\end{theorem}
\begin{remark}
	By subtracting $\tau \tilde{u}(t,\cdot)$ in the $L^2$ norm, we find the derived estimates for $u(t,\cdot)$ in Theorem \ref{THM_Asymptoti_Profiles} can be improved $t^{-\frac{1}{4}}$ if $n=1$ and $(\ln t)^{-\frac{1}{2}}$ if $n=2$ for $t\gg1$. 
	It is still open that the gained decay rate for $n\geqslant 3$.
\end{remark}
\begin{remark}
	Indeed, from Theorems \ref{THM_Asymptoti_Profiles} and \ref{THM_APPROX}, one may derive
	\begin{align*}
	\|u(t,\cdot)-\tau\tilde{u}(t,\cdot)\|_{L^2(\mb{R}^n)}&=\|\chi_{\intt}(D)(u(t,\cdot)-\tau\tilde{u}(t,\cdot))\|_{L^2(\mb{R}^n)}+\|(1-\chi_{\intt}(D))(u(t,\cdot)-\tau\tilde{u}(t,\cdot))\|_{L^2(\mb{R}^n)}\\
	&\lesssim t^{\frac{1}{2}-\frac{n}{4}}\|u_2\|_{L^{1,1}(\mb{R}^n)}+\mathrm{e}^{-ct}\|u_2\|_{L^2(\mb{R}^n)}\lesssim t^{\frac{1}{2}-\frac{n}{4}}\|u_2\|_{L^2(\mb{R}^n)\cap L^{1,1}(\mb{R}^n)}
	\end{align*}
	for $t\geqslant t_0\gg 1$. Moreover, concerning $0\leqslant t\leqslant t_0$, it is trivial that
	\begin{align*}
	\|u(t,\cdot)-\tau\tilde{u}(t,\cdot)\|_{L^2(\mb{R}^n)}\lesssim \|u(t,\cdot)\|_{L^2(\mb{R}^n)}+\tau\|\tilde{u}(t,\cdot)\|_{L^2(\mb{R}^n)}\lesssim \|u_2\|_{L^2(\mb{R}^n)\cap L^{1,1}(\mb{R}^n)}.
	\end{align*}
	Therefore, the approximate relation holds for all $t\geqslant 0$ and the whole spaces such that
	\begin{align*}
	\|u(t,\cdot)-\tau\tilde{u}(t,\cdot)\|_{L^2(\mb{R}^n)}\lesssim (1+t)^{\frac{1}{2}-\frac{n}{4}}\|u_2\|_{L^2(\mb{R}^n)\cap L^{1,1}(\mb{R}^n)},
	\end{align*}
	where we assumed $u_2\in L^2(\mb{R}^n)\cap L^{1,1}(\mb{R}^n)$. Namely, the solution for the linear MGT equation approximate to that for the linear viscoelastic damped wave equation at least for $n=1,2$.
\end{remark}
\begin{proof}
	By applying the partial Fourier transform $\hat{\tilde{u}}(t,\xi)=\ml{F}_{x\to\xi}(\tilde{u}(t,x))$, let us recall the derived inequality stated in Lemma 2.1 in \cite{Ikehata2014} such that
	\begin{align}\label{Ikehta_00002}
	\left\|\chi_{\intt}(\xi)\left(\hat{\tilde{u}}(t,\xi)-\frac{\sin(|\xi|t)}{|\xi|}\mathrm{e}^{-\frac{\beta}{2}|\xi|^2t}P_{u_2}\right)\right\|_{L^2(\mb{R}^n)}\lesssim t^{-\frac{n}{4}}\|u_2\|_{L^{1,1}(\mb{R}^n)}
	\end{align}
	for $t\gg1$. Again, $\tilde{u}(t,x)$ is the solution to the viscoelastic damped wave equation \eqref{Linear_Visco_Damped_Wave} with initial data choosing by $\tilde{u}_1(x)=u_2(x)$.
	
	We notice that the difference of the solutions can be decomposed by three components as follows:
	\begin{align*}
	\hat{u}(t,\xi)-\tau\hat{\tilde{u}}(t,\xi)&=\left(\hat{u}(t,\xi)-J(t,|\xi|)P_{u_2}\right)+\left(J(t,|\xi|)-\tau\frac{\sin(|\xi|t)}{|\xi|}\mathrm{e}^{-\frac{\beta}{2}|\xi|^2t}\right)P_{u_2}\\
	&\quad+\left(\tau\frac{\sin(|\xi|t)}{|\xi|}\mathrm{e}^{-\frac{\beta}{2}|\xi|^2t}P_{u_2}-\tau\hat{\tilde{u}}(t,\xi)\right).
	\end{align*}
	Therefore, employing the  Parseval equality and  the norm inequality, we arrive at
	\begin{align*}
	&\left\|\chi_{\intt}(D)\left(u(t,\cdot)-\tau\tilde{u}(t,\cdot)\right)\right\|_{L^2(\mb{R}^n)}=\left\|\chi_{\intt}(\xi)\left(\hat{u}(t,\xi)-\tau\hat{\tilde{u}}(t,\xi)\right)\right\|_{L^2(\mb{R}^n)}\\
	&\lesssim \left\|\chi_{\intt}(\xi)\left(\hat{u}(t,\xi)-J(t,|\xi|)P_{u_2}\right)\right\|_{L^2(\mb{R}^n)}+\left\|\chi_{\intt}(\xi)\left(J(t,|\xi|)-\tau\frac{\sin(|\xi|t)}{|\xi|}\mathrm{e}^{-\frac{\beta}{2}|\xi|^2t}\right)\right\|_{L^2(\mb{R}^n)}|P_{u_2}|\\
	&\quad+\tau\left\|\chi_{\intt}(\xi)\left(\hat{\tilde{u}}(t,\xi)-\frac{\sin(|\xi|t)}{|\xi|}\mathrm{e}^{-\frac{\beta}{2}|\xi|^2t}P_{u_2}\right)\right\|_{L^2(\mb{R}^n)}\\
	&\lesssim t^{-\frac{n}{4}}\|u_2\|_{L^{1,1}(\mb{R}^n)}+\ml{J}(t)|P_{u_2}|,
	\end{align*}
	where we used \eqref{Chen_00001} and \eqref{Ikehta_00002}. In the above inequality, we denoted
	\begin{align*}
	\ml{J}(t):=\left\|\chi_{\intt}(\xi)\left(J(t,|\xi|)-\tau\frac{\sin(|\xi|t)}{|\xi|}\mathrm{e}^{-\frac{\beta}{2}|\xi|^2t}\right)\right\|_{L^2(\mb{R}^n)}.
	\end{align*}
	In other words, we just need to estimate $\ml{J}(t)$ in the remaindering part of the proof.
	
	Recalling \eqref{Defn_T0}, from the definition of $J(t)$ in the last subsection, we may estimate
	\begin{align*}
	\ml{J}(t)&\lesssim \left\|\chi_{\intt}(\xi) \frac{\mathrm{e}^{-\frac{\beta-\tau}{2}|\xi|^2t}}{\ml{T}_0^2(|\xi|)+|\xi|^2}\left(-\cos(|\xi|t)+\mathrm{e}^{-\frac{1}{\tau}t+\frac{3}{2}(\beta-\tau)|\xi|^2t}\right)\right\|_{L^2(\mb{R}^n)}\\
	&\quad+\left\|\chi_{\intt}(\xi)\frac{\sin(|\xi|t)}{|\xi|}\mathrm{e}^{-\frac{\beta}{2}|\xi|^2t}\left(\frac{\ml{T}_0(|\xi|)\mathrm{e}^{\frac{\tau}{2}|\xi|^2t}}{\ml{T}_0^2(|\xi|)+|\xi|^2}-\tau\right)\right\|_{L^2(\mb{R}^n)}\\
	&=:\ml{J}^{(1)}(t)+\ml{J}^{(2)}(t).
	\end{align*}
	For one thing, it is clear that
	\begin{align*}
	\ml{J}^{(1)}(t)\lesssim\left\|\chi_{\intt}(\xi)|\cos(|\xi|t)|\mathrm{e}^{-c|\xi|^2t}\right\|_{L^2(\mb{R}^n)}+\mathrm{e}^{-ct}\lesssim t^{-\frac{n}{4}}
	\end{align*}
	for $t\gg1$. Before estimating $\ml{J}^{(2)}(t)$, the explicit computation shows the identity as follows:
	\begin{align*}
	\ml{T}_0(|\xi|)\mathrm{e}^{\frac{\tau}{2}|\xi|^2t}-\tau\left(\ml{T}_0^2(|\xi|)+|\xi|^2\right)&=\ml{T}_0(|\xi|)\left(\mathrm{e}^{\frac{\tau}{2}|\xi|^2t}-1+\frac{3}{2}\tau(\beta-\tau)|\xi|^2\right)-\tau|\xi|^2\\
	&=\frac{\tau}{2}|\xi|^2t\,\ml{T}_0(|\xi|)\int_0^1\mathrm{e}^{\frac{\tau}{2}|\xi|^2ts}\mathrm{d}s+\tau\left(\ml{T}_0(|\xi|)\frac{3}{2}(\beta-\tau)-1\right)|\xi|^2.
	\end{align*}
	Thus, we compute
	\begin{align*}
	\ml{J}^{(2)}(t)&\lesssim t\left\|\chi_{\intt}(\xi)|\sin(|\xi|t)|\mathrm{e}^{-\frac{\beta}{2}|\xi|^2t}|\xi|\left|\ml{T}_0(|\xi|)\right|\int_0^1\mathrm{e}^{\frac{\tau}{2}|\xi|^2ts}\mathrm{d}s\right\|_{L^2(\mb{R}^n)}\\
	&\quad+\left\|\chi_{\intt}(\xi)|\sin(|\xi|t)|\mathrm{e}^{-\frac{\beta}{2}|\xi|^2t}\left|\ml{T}_0(|\xi|)\frac{3}{2}(\beta-\tau)-1\right||\xi|\right\|_{L^2(\mb{R}^n)}\\
	&\lesssim t\left\|\chi_{\intt}(\xi)|\xi|\,|\sin(|\xi|t)|\mathrm{e}^{-c|\xi|^2t}\right\|_{L^2(\mb{R}^n)}\lesssim t^{\frac{1}{2}-\frac{n}{4}}
	\end{align*}
	for $t\gg1$. Summarizing the derived estimates, one has
	\begin{align*}
	\left\|\chi_{\intt}(D)\left(u(t,\cdot)-\tau\tilde{u}(t,\cdot)\right)\right\|_{L^2(\mb{R}^n)}\lesssim t^{-\frac{n}{4}}\|u_2\|_{L^{1,1}(\mb{R}^n)}+t^{\frac{1}{2}-\frac{n}{4}}|P_{u_2}|,
	\end{align*}
	and the proof is immediately complete.
\end{proof}
\begin{remark}
	Although we can observe a similar decay property between the linear MGT equation \eqref{Linear_MGT_Dissipative} and the linear viscoelastic damped wave equation \eqref{Linear_Visco_Damped_Wave}, there is a great difference between them. Recalling that the property of finite propagation speed (FPS) is valid for the linear MGT equation in the conservative case, we refer to Section 2 in \cite{ChenPalmieri201901}. Actually, the property of FPS holds for the linear MGT equation even in the dissipative case (see Remark \ref{Rem_Finite_Prop_Speed}). To estimate the propagation speed, we construct an energy for \eqref{Linear_MGT_Dissipative}, namely,
	\begin{align*}
	\ml{E}_{\mathrm{FPS}}[u](t)&:=\beta\int_{\Lambda_{\beta,\tau}}\left|\nabla u_t(t,x)+\frac{1}{\beta}\nabla u(t,x)\right|^2\mathrm{d}x+\tau\int_{\Lambda_{\beta,\tau}}\left|u_{tt}(t,x)+\frac{1}{\beta}u_t(t,x)\right|^2\mathrm{d}x\\
	&\quad\,\,+\frac{1}{\beta}\left(1-\frac{\tau}{\beta}\right)\int_{\Lambda_{\beta,\tau}}|u_t(t,x)|^2\mathrm{d}x,
	\end{align*}
	where the domain is defined by
	\begin{align*}
	\Lambda_{\beta,\tau}:=\left\{(t,x):t\in[0,T],\ |x-x_0|\leqslant\sqrt{\beta/\tau}\,(T-t)\right\}.
	\end{align*}
	So, taking the derivative with respect to $t$, we arrive at
	\begin{align*}
	\frac{\mathrm{d}}{\mathrm{d} t}\ml{E}_{\mathrm{FPS}}[u](t)&=2\beta\int_{\Lambda_{\beta,\tau}}\nabla\left(u_t(t,x)+\frac{1}{\beta}u(t,x)\right)\cdot\nabla\left(u_{tt}(t,x)+\frac{1}{\beta}u_t(t,x)\right)\mathrm{d}x\\
	&\quad+2\tau\int_{\Lambda_{\beta,\tau}}\left(u_{tt}(t,x)+\frac{1}{\beta}u_t(t,x)\right)\left(u_{ttt}(t,x)+\frac{1}{\beta}u_{tt}(t,x)\right)\mathrm{d}x\\
	&\quad+\frac{2}{\beta}\left(1-\frac{\tau}{\beta}\right)\int_{\Lambda_{\beta,\tau}}u_{t}(t,x)u_{tt}(t,x)\mathrm{d}x\\
	&\quad-\beta\sqrt{\frac{\beta}{\tau}}\int_{\partial\Lambda_{\beta,\tau}}\left|\nabla u_t(t,x)+\frac{1}{\beta}\nabla u(t,x)\right|^2\mathrm{d}S-\tau\sqrt{\frac{\beta}{\tau}}\int_{\partial\Lambda_{\beta,\tau}}\left|u_{tt}(t,x)+\frac{1}{\beta}u_t(t,x)\right|^2\mathrm{d}S\\
	&\quad+\frac{1}{\beta}\sqrt{\frac{\beta}{\tau}}\left(1-\frac{\tau}{\beta}\right)\int_{\partial\Lambda_{\beta,\tau}}|u_t(t,x)|^2\mathrm{d}S.
	\end{align*}
	Integration by parts and the Cauchy-Schwarz inequality yield
	\begin{align*}
	&2\beta\int_{\Lambda_{\beta,\tau}}\nabla\left(u_t(t,x)+\frac{1}{\beta}u(t,x)\right)\cdot\nabla\left(u_{tt}(t,x)+\frac{1}{\beta}u_t(t,x)\right)\mathrm{d}x\\
	&\leqslant -2\beta\int_{\Lambda_{\beta,\tau}}\left(\Delta u_t(t,x)+\frac{1}{\beta}\Delta u(t,x)\right)\left(u_{tt}(t,x)+\frac{1}{\beta}u_t(t,x)\right)\mathrm{d}x\\
	&\quad+\beta\sqrt{\frac{\beta}{\tau}}\int_{\partial\Lambda_{\beta,\tau}}\left|\nabla u_t(t,x)+\frac{1}{\beta}\nabla u(t,x)\right|^2\mathrm{d}S+\beta\sqrt{\frac{\tau}{\beta}}\int_{\partial\Lambda_{\beta,\tau}}\left|u_{tt}(t,x)+\frac{1}{\beta}u_t(t,x)\right|^2\mathrm{d}S.
	\end{align*}
	Finally, summarizing them, one has
	\begin{align*}
	\frac{\mathrm{d}}{\mathrm{d}t}\ml{E}_{\mathrm{FPS}}[u](t)\leqslant-2\left(1-\frac{\tau}{\beta}\right)\int_{\Lambda_{\beta,\tau}}|u_{tt}(t,x)|^2\mathrm{d}x-\frac{1}{\sqrt{\beta\tau}}\int_{\partial\Lambda_{\beta,\tau}}|u_t(t,x)|^2\mathrm{d}S\leqslant0.
	\end{align*}
	In other words, for all $t\in[0,T]$, it is valid that $\ml{E}_{\mathrm{FPS}}[u](t)\leqslant \ml{E}_{\mathrm{FPS}}[u](0)=0$ if $u(0,x)=u_t(0,x)=u_{tt}(0,x)=0$ in a set $\Lambda_{\beta,\tau}^0:=\left\{|x-x_0|\leqslant\sqrt{\beta/\tau}\,T\right\}$. According to the definition of energy $\ml{E}_{\mathrm{FPS}}[u](t)$, it follows immediately that $u_t(t,x)=0$ and $\nabla u(t,x)=0$ in $\Lambda_{\beta,\tau}$. This implies $u(t,x)=0$ in $\Lambda_{\beta,\tau}$ due to $u(0,x)=0$ in $\Lambda_{\beta,\tau}^0$. The propagation speed reads as $\sqrt{\beta/\tau}$. In the limit case $\tau=\beta$, the propagation speed is equal to $1$, which corresponds to the statement of Section 2 in \cite{ChenPalmieri201901}. In particular, formally taking $\tau=0$, i.e. the viscoelastic damped wave equation, the propagation speed is infinite.  However, the property of FPS does not hold anymore in the linear viscoelastic damped wave equation, and the solution to \eqref{Linear_Visco_Damped_Wave} has some smoothing effects.
\end{remark}

\begin{remark}
	Actually, there is another aspect to analyze the linear MGT equation. Let us consider the conservative case ($\tau=\beta$) in the Cauchy problem \eqref{Linear_MGT_Dissipative}, whose local (in spaces) energy with $R>0$ can be defined by
	\begin{align*}
	E_{\mathrm{MGT},R}[u](t):=\frac{1}{2}\|\partial_t(\beta u_t(t,\cdot)+u(t,\cdot))\|_{L^2(B_R)}+\frac{1}{2}\|\nabla(\beta u_t(t,\cdot)+u(t,\cdot))\|_{L^2(B_R)}.
	\end{align*}
	Motivated by \cite{ChenPalmieri201901}, we may rewrite the Cauchy problem \eqref{Linear_MGT_Dissipative} with $\tau=\beta$ by the next way:
	\begin{align*}
	\begin{cases}
	(\beta u_t+u)_{tt}-\Delta (\beta u_t+u)=0,&x\in\mb{R}^n,\ t>0,\\
	(\beta u_t+u)(0,x)=0,\ (\beta u_t+u)_t(0,x)=\beta u_2(x),&x\in\mb{R}^n.
	\end{cases}
	\end{align*}
	Then, by applying Theorem 1.2 with $c(x)\equiv 1$, $n\geqslant 3$ and $L=0$ in the recent paper \cite{Charao-Ikehata-2020}, we can get
	\begin{align*}
	E_{\mathrm{MGT},R}[u](t)=\ml{O}(t^{-1})
	\end{align*}
	for each $R>0$ and $t\gg 1$, provided that $u_2\in L^1(\mb{R}^n)$ and 
	\begin{align*}
	\int_{\mb{R}^n}(1+|x|)|u_2(x)|^2\mathrm{d}x<\infty.
	\end{align*} 
	In other words, local (in spaces) energy for the linear MGT equation in the conservative case $\tau=\beta$ decays with an algebraic decay order.
\end{remark}

\section{Singular limit problem}\label{Sec_Singu_Limit}
In this section, we focus on the following Cauchy problem for the singular limit problem of the form:
\begin{align}\label{Singular_Limit_MGT}
\begin{cases}
\tau u_{\tau,ttt}+u_{\tau,tt}-\Delta u_{\tau}-\beta\Delta u_{\tau,t}=0,&x\in\mb{R}^n, \ t>0,\\
u_{\tau}(0,x)=u_0(x),\ u_{\tau,t}(0,x)=u_1(x),\ u_{\tau,tt}(0,x)=u_2(x),&x\in\mb{R}^n,
\end{cases}
\end{align}
where $\tau\in(0,\beta)$ with $\beta>0$. Moreover, the time-derivative for the unknown $u_{\tau}=u_{\tau}(t,x)$ is denoted by $u_{\tau,t}:=\partial_tu_{\tau}$, and similarly for $u_{\tau,tt}$ as well as $u_{\tau,ttt}$. Particularly, we consider $\tau$ to be a small parameter such that $0<\tau\ll \beta$. In other words, our main purpose in the section is to understand the asymptotic profiles of the solution $u_{\tau}=u_{\tau}(t,x)$ as $\tau\to0^+$. This property has been studied between damped wave equations and heat equations. We refer readers to \cite{Kisynski1963,Ikehata2003,Ikehata-Nishihara-2003,Chill-Haraux,Hashimoto-Yamazaki,Ghisi-Gobbino,Ikehata-Sobajima} and references therein. Nevertheless, concerning the study of the Cauchy problem for the linear MGT equation, it seems new from the knowledge of authors.

Let us introduce the Cauchy problem for the viscoelastic damped wave equation, namely,
\begin{align}\label{Singular_Limit_Visco_Elastic}
\begin{cases}
v_{tt}-\Delta v-\beta\Delta v_t=0,&x\in\mb{R}^n, \ t>0,\\
v(0,x)=u_0(x),\ v_t(0,x)=u_1(x),&x\in\mb{R}^n,
\end{cases}
\end{align}
where $\beta>0$. As mentioned in the last section, the Cauchy problem for the viscoelastic damped wave equation has been widely studied. For instance, considering Theorem 14.3.2 and Corollary 14.3.1 in the book \cite{Ebert-Reissig-book}, we know solutions to the Cauchy problem \eqref{Singular_Limit_Visco_Elastic} fulfill
\begin{align*}
\left\|\,|D|^kv(t,\cdot)\right\|_{L^2(\mb{R}^n)}^2&\leqslant C\left((1+t)^{-k}\|u_0\|_{H^k(\mb{R}^n)}^2+(1+t)^{-(k-1)}\|u_1\|_{H^{k-1}(\mb{R}^n)}^2\right)\ \ \mbox{for}\ \ k\geqslant 1,\\
\left\|\,|D|^kv_t(t,\cdot)\right\|_{L^2(\mb{R}^n)}^2&\leqslant C\left((1+t)^{-(k+1)}\|u_0\|_{H^{k}(\mb{R}^n)}^2+(1+t)^{-k}\|u_1\|_{H^{k}(\mb{R}^n)}^2\right)\ \ \ \ \, \mbox{for}\ \ k\geqslant 1.
\end{align*}
Therefore, it is easy to observe that
\begin{align}
\sum\limits_{j,k=1,2}\left\|\nabla^{j}\partial_t^kv(t,\cdot)\right\|_{L^2(\mb{R}^n)}^2&\leqslant\begin{cases}
C(1+t)^{-1}\|(u_0,u_1)\|_{H^4(\mb{R}^n)\times H^4(\mb{R}^n)}^2&\mbox{if}\ \ u_1\neq0,\\
C(1+t)^{-2}\|u_0\|_{H^4(\mb{R}^n)}^2&\mbox{if}\ \ u_1=0,
\end{cases}\label{Est_Vis_Ela_Sing}\\
\sum\limits_{j,k=0,1}\left\|\nabla^{1+j}\partial_t^{k}v(t,\cdot)\right\|_{L^2(\mb{R}^n)}^2&\leqslant\begin{cases}
C\|(u_0,u_1)\|_{H^2(\mb{R}^n)\times H^2(\mb{R}^n)}^2&\mbox{if}\ \ u_1\neq0,\\
C(1+t)^{-1}\|u_0\|_{H^2(\mb{R}^n)}^2&\mbox{if}\ \ u_1=0,
\end{cases}\label{Est_Vis_Ela_Sing_2}
\end{align}
where we employed $\|\nabla^jf(t,\cdot)\|_{L^2(\mb{R}^n)}\approx\|\,|D|^jf(t,\cdot)\|_{L^2(\mb{R}^n)}$.\\
Finally, let us define $w=w(t,x)$ such that
\begin{align}\label{Difference}
w(t,x):=u_{\tau}(t,x)-v(t,x),
\end{align}
where $u_{\tau}=u_{\tau}(t,x)$ is the solution to the Cauchy problem \eqref{Singular_Limit_MGT} and $v=v(t,x)$ is the solution to the Cauchy problem \eqref{Singular_Limit_Visco_Elastic}.

\subsection{Singular limit for an energy}
\begin{theorem}\label{THM_Singular_Limit}
	Let us assume $(u_0,u_1,u_2)\in H^4(\mb{R}^n)\times H^4(\mb{R}^n)\times L^2(\mb{R}^n)$, where $u_0$ and $u_1$ are not zero simultaneously. Then, the difference $w=w(t,x)$ defined in \eqref{Difference} fulfills the following estimates for small $\tau$ such that $0<\tau\ll \beta$:
	\begin{align*}
	&\ml{E}[w](t)+(2-\varepsilon_1-2\tau k)\int_0^t\|w_{\eta\eta}(\eta,\cdot)\|_{L^2(\mb{R}^n)}^2\mathrm{d}\eta+(2\beta k-\varepsilon_1-2)\int_0^t\|\nabla w_{\eta}(\eta,\cdot)\|_{L^2(\mb{R}^n)}^2\mathrm{d}\eta\\
	&\leqslant\tau\|u_2-\Delta u_0-\beta\Delta u_1\|_{L^2(\mb{R}^n)}^2+\begin{cases}
	C\tau^2\ln(\mathrm{e}+t)\|(u_0,u_1)\|_{H^4(\mb{R}^n)\times H^4(\mb{R}^n)}^2&\mbox{if}\ \ u_1\neq0,\\
	C\tau^2\|u_0\|_{H^4(\mb{R}^n)}^2&\mbox{if}\ \ u_1=0,
	\end{cases}
	\end{align*}
	where $C$ is a positive constant independent of $\tau$, and $k\in \left[\frac{2+\varepsilon_1}{2\beta},\frac{2-\varepsilon_1}{2\tau}\right]$ carrying  $\varepsilon_1\in\left(0,\frac{2\beta-2\tau}{\beta+\tau}\right]$. Moreover, the energy $\ml{E}[w](t)$ is defined by
	\begin{align*}
	\ml{E}[w](t)&:=\beta\left\|\nabla w_t(t,\cdot)+\frac{1}{\beta}\nabla w(t,\cdot)\right\|_{L^2(\mb{R}^n)}^2+\tau\|w_{tt}(t,\cdot)+kw_t(t,\cdot)\|_{L^2(\mb{R}^n)}^2\\
	&\quad\,\, +k(1-\tau k)\|w_t(t,\cdot)\|_{L^2(\mb{R}^n)}^2+\left(k-\frac{1}{\beta}\right)\|\nabla w(t,\cdot)\|_{L^2(\mb{R}^n)}^2.
	\end{align*}
\end{theorem}
\begin{remark}
	The assumption that $u_0$ and $u_1$ are not zero simultaneously, is natural to guarantee nontrivial solution for the viscoelastic damped wave equation \eqref{Singular_Limit_Visco_Elastic}.
\end{remark}
\begin{remark}\label{Rem_Indpendent}
	Actually, the choice of parameters $k$ and $\varepsilon_1$ can be independent of $\tau$ for a small value of $\tau>0$. For example, by taking a small $\tau$ such that $0<\tau\leqslant 39\beta/41$, we can fix $\varepsilon_1=1/20$ and $k=41/(40\beta)$. Particularly, by considering $\tau\to0^+$, we can immediately enlarge the choice of parameters $k$ and $\varepsilon_1$.
\end{remark}
\begin{remark}
	Let us consider $t\in(0,T)$. In Theorem \ref{THM_Singular_Limit} with $T<\infty$, we may observe
	\begin{itemize}
		\item if $u_2\neq\Delta u_0+\beta\Delta u_1$, it holds $\ml{E}[w](t)=\ml{O}(\tau)$ as $\tau\to0^+$;
		\item if $u_2=\Delta u_0+\beta\Delta u_1$, it holds $\ml{E}[w](t)=\ml{O}(\tau^2)$ as $\tau\to0^+$.
	\end{itemize}
	So, the speeds of convergence are different under different choices of initial data.
	However, concerning $T=\infty$, the property for singular limit holds if and only if $u_1=0$. Otherwise, we found that $\ml{E}[w](t)=\ml{O}(\ln(\mathrm{e}+t))$ as $t\to\infty$. In conclusion, the choice for initial data is extremely important in the consideration of singular limit property.
\end{remark}
\begin{proof}
	Let us act $\tau\partial_t$ on the equation in \eqref{Singular_Limit_Visco_Elastic} and then add itself to arrive at
	\begin{align}\label{New_Eq_Visco_Elastic}
	\tau v_{ttt}+v_{tt}-\Delta v-\beta\Delta v_t=\tau(\Delta v_t+\beta\Delta v_{tt}).
	\end{align}
	Let us recall $w=w(t,x)$ as a difference such that $w(t,x)=u_{\tau}(t,x)-v(t,x)$. Then, by subtracting the equation in \eqref{Singular_Limit_MGT} with \eqref{New_Eq_Visco_Elastic}, we have
	\begin{align}\label{New_Eq_Diff}
	\tau w_{ttt}+w_{tt}-\Delta w-\beta \Delta w_t=-\tau(\Delta v_t+\beta \Delta v_{tt}).
	\end{align}
	To achieve our aim, we next will apply the classical energy method for the Cauchy problem. For one thing, we construct an energy as follows:
	\begin{align*}
	\ml{E}_1[w](t):=\tau\|w_{tt}(t,\cdot)\|_{L^2(\mb{R}^n)}^2+\beta\|\nabla w_t(t,\cdot)\|_{L^2(\mb{R}^n)}^2-2\int_{\mb{R}^n}\Delta w(t,x)w_t(t,x)\mathrm{d}x.
	\end{align*}
	It shows that
	\begin{align*}
	\frac{\mathrm{d}}{\mathrm{d}t}\ml{E}_1[w](t)&=2\tau\int_{\mb{R}^n}w_{ttt}(t,x)w_{tt}(t,x)\mathrm{d}x+2\beta\int_{\mb{R}^n}\nabla w_{tt}(t,x)\cdot\nabla w_t(t,x)\mathrm{d}x\\
	&\quad-2\int_{\mb{R}^n}\Delta w_t(t,x)w_t(t,x)\mathrm{d}x-2\int_{\mb{R}^n}\Delta w(t,x)w_{tt}(t,x)\mathrm{d}x\\
	&=-2\|w_{tt}(t,\cdot)\|_{L^2(\mb{R}^n)}^2+2\|\nabla w_t(t,\cdot)\|_{L^2(\mb{R}^n)}^2-2\tau\int_{\mb{R}^n}(\Delta v_t(t,x)+\beta\Delta v_{tt}(t,x))w_{tt}(t,x)\mathrm{d}x,
	\end{align*}
	where we considered \eqref{New_Eq_Diff}.\\
	For another, let us introduce the other auxiliary energy
	\begin{align*}
	\ml{E}_2[w](t):=\|\nabla w(t,\cdot)\|_{L^2(\mb{R}^n)}^2+\|w_t(t,\cdot)\|_{L^2(\mb{R}^n)}^2+2\tau\int_{\mb{R}^n}w_{tt}(t,x)w_t(t,x)\mathrm{d}x.
	\end{align*}
	Taking the derivative with respect to time variable, we have
	\begin{align*}
	\frac{\mathrm{d}}{\mathrm{d}t}\ml{E}_2[w](t)&=2\int_{\mb{R}^n}\nabla w_t(t,x)\cdot\nabla w(t,x)\mathrm{d}x+2\int_{\mb{R}^n}w_{tt}(t,x)w_t(t,x)\mathrm{d}x\\
	&\quad+2\tau\int_{\mb{R}^n}w_{ttt}(t,x)w_t(t,x)\mathrm{d}x+2\tau\int_{\mb{R}^n}w_{tt}(t,x)w_{tt}(t,x)\mathrm{d}x\\
	&=-2\beta\|\nabla w_t(t,\cdot)\|_{L^2(\mb{R}^n)}^2+2\tau\|w_{tt}(t,\cdot)\|_{L^2(\mb{R}^n)}^2\\
	&\quad+2\tau\int_{\mb{R}^n}(\nabla v_t(t,x)+\beta\nabla v_{tt}(t,x))\cdot\nabla w_t(t,x)\mathrm{d}x.
	\end{align*}
	Let $k$ be a positive parameter to be fixed later. Hence, by applying
	\begin{align*}
	&-2\int_{\mb{R}^n}\Delta w(t,x)w_t(t,x)\mathrm{d}x=2\int_{\mb{R}^n}\nabla w(t,x)\cdot\nabla w_t(t,x)\mathrm{d}x\\
	&=\beta\left\|\frac{1}{\beta}\nabla w(t,\cdot)+\nabla w_t(t,\cdot)\right\|_{L^2(\mb{R}^n)}^2-\left(\frac{1}{\beta}\|\nabla w(t,\cdot)\|_{L^2(\mb{R}^n)}^2+\beta\|\nabla w_t(t,\cdot)\|_{L^2(\mb{R}^n)}^2\right)
	\end{align*}
	and
	\begin{align*}
	2\tau k\int_{\mb{R}^n}w_{tt}(t,x)w_t(t,x)\mathrm{d}x&=\tau\|w_{tt}(t,\cdot)+kw_t(t,\cdot)\|_{L^2(\mb{R}^n)}^2-\tau\left(\|w_{tt}(t,\cdot)\|_{L^2(\mb{R}^n)}^2+k^2\|w_t(t,\cdot)\|_{L^2(\mb{R}^n)}^2\right),
	\end{align*}
	one may rewrite the total energy by
	\begin{align*}
	\ml{E}_1[w](t)+k\ml{E}_2[w](t)&=\beta\left\|\nabla w_t(t,\cdot)+\frac{1}{\beta}\nabla w(t,\cdot)\right\|_{L^2(\mb{R}^n)}^2+\tau\|w_{tt}(t,\cdot)+kw_t(t,\cdot)\|_{L^2(\mb{R}^n)}^2\\
	&\quad +k(1-\tau k)\|w_t(t,\cdot)\|_{L^2(\mb{R}^n)}^2+\left(k-\frac{1}{\beta}\right)\|\nabla w(t,\cdot)\|_{L^2(\mb{R}^n)}^2.
	\end{align*}
	To guarantee the non-negativity of the above combined energy, we need to restrict $k\in\left[\frac{1}{\beta},\frac{1}{\tau}\right]$. Here, we should underline in advance that
	\begin{align*}
	\ml{E}_1[w](0)+k\ml{E}_2[w](0)=\tau\|u_2-\Delta u_0-\beta\Delta u_1\|_{L^2(\mb{R}^n)}^2,
	\end{align*}
	since $w(0,x)=w_t(0,x)=0$ and $w_{tt}(0,x)=u_2(x)-\Delta u_0(x)-\beta\Delta u_1(x)$.
	
	Furthermore, the application of Cauchy's inequality indicates that there exists a small constant $\varepsilon_1>0$ such that
	\begin{align*}
	&-2\tau\int_{\mb{R}^n}(\Delta v_{t}(t,x)+\beta \Delta v_{tt}(t,x))w_{tt}(t,x)\mathrm{d}x+2k\tau\int_{\mb{R}^n}(\nabla v_t(t,x)+\beta\nabla v_{tt}(t,x))\cdot\nabla w_t(t,x)\mathrm{d}x\\
	&\leqslant\frac{2\tau^2}{\varepsilon_1}\left(\|\Delta v_t(t,\cdot)\|_{L^2(\mb{R}^n)}^2+\beta^2\|\Delta v_{tt}(t,\cdot)\|_{L^2(\mb{R}^n)}^2+k^2\left(\|\nabla v_t(t,\cdot)\|_{L^2(\mb{R}^n)}^2+\beta^2\|\nabla v_{tt}(t,\cdot)\|_{L^2(\mb{R}^n)}^2\right)\right)\\
	&\quad+\varepsilon_1\left(\|w_{tt}(t,\cdot)\|_{L^2(\mb{R}^n)}^2+\|\nabla w_{t}(t,\cdot)\|_{L^2(\mb{R}^n)}^2\right).
	\end{align*}
	We summarize the derived inequalities, which lead to
	\begin{align*}
	&\frac{\mathrm{d}}{\mathrm{d}t}\left(\ml{E}_1[w](t)+k\ml{E}_2[w](t)\right)+(2-\varepsilon_1-2\tau k)\|w_{tt}(t,\cdot)\|_{L^2(\mb{R}^n)}^2+(2\beta k-\varepsilon_1-2)\|\nabla w_t(t,\cdot)\|_{L^2(\mb{R}^n)}^2\\
	&\leqslant \frac{2\tau^2}{\varepsilon_1}\left(\|\Delta v_t(t,\cdot)\|_{L^2(\mb{R}^n)}^2+k^2\|\nabla v_t(t,\cdot)\|_{L^2(\mb{R}^n)}^2+\beta^2\left(\|\Delta v_{tt}(t,\cdot)\|_{L^2(\mb{R}^n)}^2+k^2\|\nabla v_{tt}(t,\cdot)\|_{L^2(\mb{R}^n)}^2\right)\right).
	\end{align*}
	By choosing $k\in \left[\frac{2+\varepsilon_1}{2\beta},\frac{2-\varepsilon_1}{2\tau}\right]$, we found that $2-\varepsilon_1-2\tau k\geqslant 0$ and $2\beta k-\varepsilon_1-2\geqslant 0$.
	To guarantee the non-empty set of $k$, we restrict ourselves $\varepsilon_1\in\left(0,\frac{2\beta-2\tau}{\beta+\tau}\right]$.\\
	Using the derived $L^2$ estimates \eqref{Est_Vis_Ela_Sing}, we see
	\begin{align*}
	&\frac{\mathrm{d}}{\mathrm{d}t}\left(\ml{E}_1[w](t)+k\ml{E}_2[w](t)\right)+(2-\varepsilon_1-2\tau k)\|w_{tt}(t,\cdot)\|_{L^2(\mb{R}^n)}^2+(2\beta k\varepsilon_1-2)\|\nabla w_t(t,\cdot)\|_{L^2(\mb{R}^n)}^2\\
	&\leqslant
	\begin{cases}
	C\tau^2(1+t)^{-1}\|(u_0,u_1)\|_{H^4(\mb{R}^n)\times H^4(\mb{R}^n)}^2&\mbox{if}\ \ u_1\neq0,\\
	C\tau^2(1+t)^{-2}\|u_0\|_{H^4(\mb{R}^n)}^2&\mbox{if}\ \ u_1=0,
	\end{cases}
	\end{align*}
	where $C$ is a positive constant independent of $\tau$. Finally, integrating the above inequality over $[0,t]$, one gets our desired inequality.
\end{proof}

\subsection{Singular limit for the solution itself}
Let us turn to the single limit for the solution itself, which is not a trivial generalization of the last result because the $L^2$ norm for the solution itself is not included in an energy of the MGT equation in the dissipative case. Motivated by \cite{Ikehata2003,Ikehata-Nishihara-2003}, we will use Hardy's inequality associated with a new variable to overcome the difficulty.
\begin{theorem}
	Let $n\geqslant 3$. Let us assume $(u_0,u_1,u_2)\in H^2(\mb{R}^n)\times H^2(\mb{R}^n)\times L^{2}(\mb{R}^n)$ and additionally $|x|u_2\in L^2(\mb{R}^n)$, where $u_0$ and $u_1$ are not zero simultaneously. Then, the difference $w(t,x)$ defined in \eqref{Difference} fulfills the following estimates for small $\tau$ such that $0<\tau\ll\beta$:
	\begin{align*}
	&\bar{C}\|w(t,\cdot)\|_{L^2(\mb{R}^n)}^2+(2\tilde{k}\beta-\varepsilon_2-2)\int_0^t\|\nabla w(\eta,\cdot)\|_{L^2(\mb{R}^n)}^2\mathrm{d}\eta+(2-\varepsilon_2-2\tilde{k}\tau)\int_0^t\|w_{\eta}(\eta,\cdot)\|_{L^2(\mb{R}^n)}^2\mathrm{d}\eta\\
	&\leqslant C\tau^2\left(\|u_2\|_{L^2(\mb{R}^n)}^2+\|\,|x|u_2\|_{L^2(\mb{R}^n)}^2\right)+ \begin{cases}
	C\tau^2t\|(u_0,u_1)\|_{H^2(\mb{R}^n)\times H^2(\mb{R}^n)}^2&\mbox{if}\ \ u_1\neq0,\\
	C\tau^2\ln(\mathrm{e}+t)\|u_0\|_{H^2(\mb{R}^n)}^2&\mbox{if}\ \ u_1=0,
	\end{cases}
	\end{align*}
	where $C,\bar{C}$ are positive constants independent of $\tau$, and $\tilde{k}\in \left[\frac{2+\varepsilon_2}{2\beta},\frac{2-\varepsilon_2}{2\tau}\right]$ carrying  $\varepsilon_2\in\left(0,\frac{2\beta-2\tau}{\beta+\tau}\right]$ and $\varepsilon_2<2$. 
\end{theorem}
\begin{remark}
	In the case when $u_1\neq0$, by using Theorem 2.1 in \cite{Ikehata-2001},  we still can provide the estimate with $\ln(\mathrm{e}+t)$ rather than $t$, where we need to assume the additional condition $\|(1+|x|)(u_1-\Delta u_0)\|_{L^2(\mb{R}^n)}<\infty$.
\end{remark}
\begin{remark}
	In the remaindering case for $n=1,2$, we may use the integral formula $w(t,x)=\int_0^tw_{\eta}(\eta,x)\mathrm{d}\eta$ with $w(0,x)=0$. Then, by applying Minkowski's integral inequality and the derived inequality in Theorem \ref{THM_Singular_Limit}, we have
	\begin{align*}
	&\|w(t,\cdot)\|_{L^2(\mb{R}^n)}^2=\int_{\mb{R}^n}\left|\int_0^tw_{\eta}(\eta,x)\mathrm{d}\eta\right|^2\mathrm{d}x\leqslant\left(\int_0^t\|w_{\eta}(\eta,\cdot)\|_{L^2(\mb{R}^n)}\mathrm{d}\eta\right)^2\\
	&\leqslant C\tau t^2\|u_2-\Delta u_0-\beta\Delta u_1\|_{L^2(\mb{R}^n)}^2+\begin{cases}
	C\tau^2\left(\int_0^t(\ln(\mathrm{e}+\eta))^{\frac{1}{2}}\mathrm{d}\eta\right)^2\|(u_0,u_1)\|_{H^4(\mb{R}^n)\times H^4(\mb{R}^n)}^2&\mbox{if}\ \ u_1\neq0,\\
	C\tau^2t^2\|u_0\|_{H^4(\mb{R}^n)}^2&\mbox{if}\ \ u_1=0,
	\end{cases}
	\end{align*}
	providing that we assume $(u_0,u_1,u_2)\in H^4(\mb{R}^n)\times H^4(\mb{R}^n)\times L^2(\mb{R}^n)$.
\end{remark}
\begin{proof}
	To begin with the proof, we introduce  $W=W(t,x)$ fulfilling
	\begin{align*}
	W(t,x):=\int_0^tw(\eta,x)\mathrm{d}\eta.
	\end{align*}
	Then, carrying out direct computations, we find that the new variable $W(t,x)$ also fulfills a kind of inhomogeneous linear MGT equation in the dissipative case. Precisely, it holds
	\begin{align*}
	&\tau W_{ttt}+W_{tt}-\Delta W-\beta \Delta W_t=\tau w_{tt}+w_t-\int_0^t\Delta w(\eta,x)\mathrm{d}\eta-\beta\Delta w\\
	&=\tau w_{tt}+w_t-\int_0^t(\tau w_{\eta\eta\eta}+w_{\eta\eta}-\beta\Delta w_{\eta}+\tau(\Delta v_{\eta}+\beta\Delta v_{\eta\eta}))(\eta,x)\mathrm{d}\eta-\beta\Delta w,
	\end{align*}
	where we applied \eqref{New_Eq_Diff}. In other words, one has
	\begin{align}\label{NN_Eq_W_Diff}
	\tau W_{ttt}+W_{tt}-\Delta W-\beta \Delta W_t=\tau u_2-\tau\Delta v-\tau\beta\Delta v_t,
	\end{align}
	since $w(0,x)=w_t(0,x)=0$ and $w_{tt}(0,x)=u_2(x)-\Delta u_0(x)-\beta\Delta u_1(x)$.
	
	Let us set two auxiliary energies as follows:
	\begin{align*}
	\widetilde{\ml{E}}_1[W](t)&:=\tau\|W_{tt}(t,\cdot)\|_{L^2(\mb{R}^n)}^2+\beta\|\nabla W_t(t,\cdot)\|_{L^2(\mb{R}^n)}^2+2\int_{\mb{R}^n}\nabla W(t,x)\cdot\nabla W_t(t,x)\mathrm{d}x,\\
	\widetilde{\ml{E}}_2[W](t)&:=\|W_t(t,\cdot)\|_{L^2(\mb{R}^n)}^2+\|\nabla W(t,\cdot)\|_{L^2(\mb{R}^n)}^2+2\tau\int_{\mb{R}^n}W_{tt}(t,x)W_t(t,x)\mathrm{d}x.
	\end{align*}
	Clearly, from integration by parts and Cauchy's inequality, we see
	\begin{align*}
	&-2\tau\int_{\mb{R}^n}(\Delta v(t,x)+\beta\Delta v_t(t,x))W_{tt}(t,x)\mathrm{d}x-2\tilde{k}\tau\int_{\mb{R}^n}(\Delta v(t,x)+\beta\Delta v_t(t,x))W_t(t,x)\mathrm{d}x\\
	&\leqslant\frac{2\tau^2}{\varepsilon_2}\left(\|\Delta v(t,\cdot)\|_{L^2(\mb{R}^n)}^2+\tilde{k}^2\|\nabla v(t,\cdot)\|_{L^2(\mb{R}^n)}^2+\beta^2\left(\|\Delta v_t(t,\cdot)\|_{L^2(\mb{R}^n)}^2+\tilde{k}^2\|\nabla v_t(t,\cdot)\|_{L^2(\mb{R}^n)}^2\right)\right)\\
	&\quad+\varepsilon_2\left(\|W_{tt}(t,\cdot)\|_{L^2(\mb{R}^n)}^2+\|\nabla W_t(t,\cdot)\|_{L^2(\mb{R}^n)}^2\right),
	\end{align*}
	where we set $\varepsilon_2\in\left(0,\frac{2\beta-2\tau}{\beta+\tau}\right]$. Here, $\tilde{k}$ is a positive constant to be restricted later.
	We now apply the similar procedure to those in the proof of Theorem \ref{THM_Singular_Limit}, then from \eqref{NN_Eq_W_Diff} we may obtain
	\begin{align*}
	\widetilde{\ml{E}}_1[W](t)+\tilde{k}\widetilde{\ml{E}}_2[W](t)&=\beta\left\|\nabla W_t(t,\cdot)+\frac{1}{\beta}\nabla W(t,\cdot)\right\|_{L^2(\mb{R}^n)}^2+\tau\|W_{tt}(t,\cdot)+\tilde{k}W_{t}(t,\cdot)\|_{L^2(\mb{R}^n)}^2\\
	&\quad+\tilde{k}(1-\tau\tilde{k})\|W_t(t,\cdot)\|_{L^2(\mb{R}^n)}^2+\left(\tilde{k}-\frac{1}{\beta}\right)\|\nabla W(t,\cdot)\|_{L^2(\mb{R}^n)}^2,
	\end{align*}
	and
	\begin{align*}
	&\frac{\mathrm{d}}{\mathrm{d}t}\left(\widetilde{\ml{E}}_1[W](t)+\tilde{k}\widetilde{\ml{E}}_2[W](t)\right)\\
	&\leqslant(2\tilde{k}\tau+\varepsilon_2-2)\|W_{tt}(t,\cdot)\|_{L^2(\mb{R}^n)}^2+(2+\varepsilon_2-2\tilde{k}\beta)\|\nabla W_t(t,\cdot)\|_{L^2(\mb{R}^n)}^2\\
	&\quad+\frac{2\tau^2}{\varepsilon_2}\left(\|\Delta v(t,\cdot)\|_{L^2(\mb{R}^n)}^2+\tilde{k}^2\|\nabla v(t,\cdot)\|_{L^2(\mb{R}^n)}^2+\beta^2\left(\|\Delta v_t(t,\cdot)\|_{L^2(\mb{R}^n)}^2+\tilde{k}^2\|\nabla v_t(t,\cdot)\|_{L^2(\mb{R}^n)}^2\right)\right)\\
	&\quad+2\tau\frac{\mathrm{d}}{\mathrm{d}t}\left(\int_{\mb{R}^n}u_2(x)(W_t(t,x)+\tilde{k}W(t,x))\mathrm{d}x\right).
	\end{align*}
	Due to the estimates \eqref{Est_Vis_Ela_Sing_2}, we observe that
	\begin{align*}
	&\frac{\mathrm{d}}{\mathrm{d}t}\left(\widetilde{\ml{E}}_1[W](t)+\tilde{k}\widetilde{\ml{E}}_2[W](t)\right)+(2-\varepsilon_2-2\tilde{k}\tau)\|W_{tt}(t,\cdot)\|_{L^2(\mb{R}^n)}^2+(2\tilde{k}\beta-\varepsilon_2-2)\|\nabla W_t(t,\cdot)\|_{L^2(\mb{R}^n)}^2\\
	&\leqslant 2\tau\frac{\mathrm{d}}{\mathrm{d}t}\left(\int_{\mb{R}^n}u_2(x)(W_t(t,x)+\tilde{k}W(t,x))\mathrm{d}x\right)+\begin{cases}
	C\tau^2\|(u_0,u_1)\|_{H^2(\mb{R}^n)\times H^2(\mb{R}^n)}^2&\mbox{if}\ \ u_1\neq0,\\
	C\tau^2(1+t)^{-1}\|u_0\|_{H^2(\mb{R}^n)}^2&\mbox{if}\ \ u_1=0.
	\end{cases}
	\end{align*}
	According to $W_{tt}(0,x)=w_t(0,x)=0$, we get $\widetilde{\ml{E}}_1[W](0)+\tilde{k}\widetilde{\ml{E}}_2[W](0)=0$.
	Integrating the previous inequality over $[0,t]$ yields
	\begin{align}\label{Est_Sing_01}
	&\beta\left\|\nabla W_t(t,\cdot)+\frac{1}{\beta}\nabla W(t,\cdot)\right\|_{L^2(\mb{R}^n)}^2+\tau\|W_{tt}(t,\cdot)+\tilde{k}W_{t}(t,\cdot)\|_{L^2(\mb{R}^n)}^2+\tilde{k}(1-\tau\tilde{k})\|W_t(t,\cdot)\|_{L^2(\mb{R}^n)}^2\notag\\
	&+\left(\tilde{k}-\frac{1}{\beta}\right)\|\nabla W(t,\cdot)\|_{L^2(\mb{R}^n)}^2+(2-\varepsilon_2-2\tilde{k}\tau)\int_0^t\|W_{\eta\eta}(\eta,\cdot)\|_{L^2(\mb{R}^n)}^2\mathrm{d}\eta\notag\\
	&+(2\tilde{k}\beta-\varepsilon_2-2)\int_0^t\|\nabla W_{\eta}(\eta,\cdot)\|_{L^2(\mb{R}^n)}^2\mathrm{d}\eta\notag\\
	&\leqslant 2\tau\int_{\mb{R}^n}u_2(x)(W_t(t,x)+\tilde{k}W(t,x))\mathrm{d}x+\begin{cases}
	C\tau^2t\|(u_0,u_1)\|_{H^2(\mb{R}^n)\times H^2(\mb{R}^n)}^2&\mbox{if}\ \ u_1\neq0,\\
	C\tau^2\ln(\mathrm{e}+t)\|u_0\|_{H^2(\mb{R}^n)}^2&\mbox{if}\ \ u_1=0.
	\end{cases}
	\end{align}
	Let us now estimate the first term on the right-hand side of \eqref{Est_Sing_01}. For one thing, there exists a positive constant $\varepsilon_3$ such that
	\begin{align*}
	2\tau\int_{\mb{R}^n}u_2(x)W_t(t,x)\mathrm{d}x\leqslant\frac{\tau^2}{\varepsilon_3}\|u_2\|_{L^2(\mb{R}^n)}^2+\varepsilon_3\|W_t(t,\cdot)\|_{L^2(\mb{R}^n)}^2.
	\end{align*}
	For another, making use of Hardy's inequality for $n\geqslant 3$, we get
	\begin{align*}
	2\tilde{k}\tau\int_{\mb{R}^n}u_2(x)W(t,x)\mathrm{d}x&\leqslant\frac{\tilde{k}\tau^2}{\varepsilon_3}\|\,|x|u_2\|_{L^2(\mb{R}^n)}^2+\varepsilon_3\tilde{k}\int_{\mb{R}^n}\frac{|W(t,x)|^2}{|x|^2}\mathrm{d}x\\
	&\leqslant\frac{\tilde{k}\tau^2}{\varepsilon_3}\|\,|x|u_2\|_{L^2(\mb{R}^n)}^2+\frac{n}{n-2}\varepsilon_3\tilde{k}\|\nabla W(t,\cdot)\|_{L^2(\mb{R}^n)}^2.
	\end{align*}
	All in all, we derive
	\begin{align*}
	&\beta\left\|\nabla W_t(t,\cdot)+\frac{1}{\beta}\nabla W(t,\cdot)\right\|_{L^2(\mb{R}^n)}^2+\tau\|W_{tt}(t,\cdot)+\tilde{k}W_{t}(t,\cdot)\|_{L^2(\mb{R}^n)}^2\\
	&+(\tilde{k}-\tau\tilde{k}^2-\varepsilon_3)\|W_t(t,\cdot)\|_{L^2(\mb{R}^n)}^2\notag+\left(\tilde{k}-\frac{1}{\beta}-\frac{n}{n-2}\varepsilon_3\tilde{k}\right)\|\nabla W(t,\cdot)\|_{L^2(\mb{R}^n)}^2\\
	&+2\left(1-\frac{\varepsilon_2}{2}-\tilde{k}\tau\right)\int_0^t\|W_{\eta\eta}(\eta,\cdot)\|_{L^2(\mb{R}^n)}^2\mathrm{d}\eta+2\left(\tilde{k}\beta-\frac{\varepsilon_2}{2}-1\right)\int_0^t\|\nabla W_{\eta}(\eta,\cdot)\|_{L^2(\mb{R}^n)}^2\mathrm{d}\eta\\
	&\leqslant \frac{C\tau^2}{\varepsilon_3}\left(\|u_2\|_{L^2(\mb{R}^n)}^2+\|\,|x|u_2\|_{L^2(\mb{R}^n)}^2\right)+\begin{cases}
	C\tau^2t\|(u_0,u_1)\|_{H^2(\mb{R}^n)\times H^2(\mb{R}^n)}^2&\mbox{if}\ \ u_1\neq0,\\
	C\tau^2\ln(\mathrm{e}+t)\|u_0\|_{H^2(\mb{R}^n)}^2&\mbox{if}\ \ u_1=0.
	\end{cases}
	\end{align*}
	
	Eventually, we just need to discuss the nonnegativity of coefficients for the norms. In the above, we need to restrict $\tilde{k}$ such that
	\begin{align*}
	1-\frac{\varepsilon_2}{2}-\tilde{k}\tau\geqslant0 \ \ \mbox{and}\ \ \tilde{k}\beta-\frac{\varepsilon_2}{2}-1\geqslant0, \ \ \mbox{iff}\ \ \tilde{k}\in\left[\frac{1+\varepsilon_2/2}{\beta},\frac{1-\varepsilon_2/2}{\tau}\right].
	\end{align*}
	Let us choose a small constant $\varepsilon_3>0$ satisfying
	\begin{align*}
	\tilde{k}-\tau\tilde{k}^2-\varepsilon_3>0\ \ \mbox{and} \ \ \left(1-\frac{n}{n-2}\varepsilon_3\right)\tilde{k}-\frac{1}{\beta}>0.
	\end{align*}
	Namely, we can determine small constant $\varepsilon_3$ such that
	\begin{align*}
	\tilde{k}\in\left[\frac{1+\varepsilon_2/2}{\beta},\frac{1-\varepsilon_2/2}{\tau}\right]\subset\left(\frac{(n-2)/(n-2-n\varepsilon_3)}{\beta},\frac{1/2+\sqrt{1/4-\tau\varepsilon_3}}{\tau}\right).
	\end{align*}
	So, the set of $\tilde{k}$ is not empty, providing that additional assumption $\varepsilon_2<2$ and
	\begin{align*}
	0<\varepsilon_3<\min\left\{\frac{n-2}{n},\frac{1}{4\tau},\frac{\varepsilon_2}{2+\varepsilon_2},\frac{2\varepsilon_2-\varepsilon_2^2}{4\tau}\right\},
	\end{align*}
	hold for $n\geqslant 3$. Indeed, the choice for these parameters can be independent of $\tau$. Let us give an example. Similarly to Remark \ref{Rem_Indpendent}, in the case of small $\tau$ such that $0<\tau\leqslant \min\{39\beta/41,1\}$, we may choose $\varepsilon_2=1/20$, $\tilde{k}=41/(40\beta)$ and $\varepsilon_3=1/1600$. Providing that $\tau\to0^+$, one may enlarge the choices of $\tilde{k}$, $\varepsilon_2$, $\varepsilon_3$.
	Recalling the relation
	\begin{align*}
	W_t(t,x)=w(t,x)=u_{\tau}(t,x)-v(t,x),
	\end{align*}  we immediately conclude our result.
\end{proof}

\section{Global (in time) existence of small data Sobolev solutions}\label{Section_Global_Existence}
\subsection{Philosophy of the proof}
According to Section \ref{Section_Linear_MGT}, we may represent the solution to the linear MGT equation in dissipative case by the form
\begin{align*}
u^{\lin}(t,x):=K(t,x)\ast_{(x)}u_2(x),
\end{align*}
where the partial Fourier transform of $K(t,x)$ with respect to $x$ was defined in \eqref{Representation_Fourier_u}. Furthermore, some $L^2$ estimates have been obtained.  In Theorem \ref{Thm_Lm-L2_Est}, by choosing $m=1$, we see
\begin{align*}
\|\,|D|^su^{\lin}(t,\cdot)\|_{L^2(\mb{R}^n)}\lesssim\tilde{g}_{n,s}(t)\|u_2\|_{L^2(\mb{R}^n)\cap L^1(\mb{R}^n)},
\end{align*}
where the time-dependent coefficient is given by
\begin{align*}
\tilde{g}_{n,s}(t):=\begin{cases}
(\ln(\mathrm{e}+t))^{\frac{1}{2}}&\mbox{if}\ \ n=2,s=0,\\
(1+t)^{\frac{1-5s}{6}}&\mbox{if}\ \ n=2,s\in(0,1/2),\\
(1+t)^{-\frac{s}{2}}&\mbox{if}\ \ n=2,s\in[1/2,2],\\
(1+t)^{\frac{1}{2}-\frac{s}{2}-\frac{n}{4}}&\mbox{if}\ \ n\geqslant 3,s\in[0,2].
\end{cases}
\end{align*}
Particularly, we denote $g_n(t):=\tilde{g}_{n,0}(t)$. Moreover, from Theorem \ref{Thm_L2-L2_Est}, one observes
\begin{align*}
\|\,|D|^su^{\lin}(t,\cdot)\|_{L^2(\mb{R}^n)}\lesssim h_s(t)\|u_2\|_{L^2(\mb{R}^n)},
\end{align*}
where the time-dependent coefficient is given by
\begin{align*}
h_s(t):=\begin{cases}
(1+t)^{1-\frac{s}{2}}&\mbox{if}\ \ s\in[0,1),\\
(1+t)^{\frac{1}{2}-\frac{s}{2}}&\mbox{if}\ \ s\in[1,2].
\end{cases}
\end{align*}
For $T>0$, we introduce the operator $N$ such that
\begin{align*}
N:\ u\in X_s(T)\to Nu(t,x):=u^{\lin}(t,x)+u^{\non}(t,x),
\end{align*}
where $X_s(T)$ is an evolution space such that
\begin{align}\label{Evolution_Spaces}
X_s(T):=\ml{C}([0,T],H^s(\mb{R}^n)),
\end{align}
with some suitable positive constants $s$ to be fixed later, and the integral operator is denoted by
\begin{align*}
u^{\non}(t,x):=\int_0^tK(t-\sigma,x)\ast_{(x)}|u(\sigma,x)|^p\mathrm{d}\sigma,
\end{align*}
which is motivated by Duhamel's principle.

In the forthcoming parts, we are going to demonstrate global (in time) existence of small data Sobolev solutions to the semilinear MGT equation \eqref{Semi_Linear_MGT_Dissipative} by proving a fixed point of operator $N$ which means $Nu\in X_s(T)$. In other words, the next two crucial inequalities:
\begin{align}
\|Nu\|_{X_s(T)}&\lesssim\|u_2\|_{L^2(\mb{R}^n)\cap L^1(\mb{R}^n)}+\|u\|_{X_s(T)}^p,\label{Cruc-01}\\
\|Nu-Nv\|_{X_s(T)}&\lesssim\|u-v\|_{X_s(T)}\left(\|u\|_{X_s(T)}^{p-1}+\|v\|_{X_s(T)}^{p-1}\right),\label{Cruc-02}
\end{align}
will be proved. Throughout this section, $u$ and $v$ are two solutions to the semilinear MGT equation \eqref{Semi_Linear_MGT_Dissipative}. Precisely, if we assume $\|u_2\|_{L^2(\mb{R}^n)\cap L^1(\mb{R}^n)}=\epsilon$ to be a sufficiently small constant, then we together \eqref{Cruc-01} with \eqref{Cruc-02} to conclude that there exists a uniquely determined local (in time) large data and global (in time) small data solution $u^{*}=u^{*}(t,x)$ belonging to the Sobolev space $X_s(T)$ by using Banach's fixed point theorem.

To end this subsection, we recall the fractional Gagliardo-Nirenberg inequality, whose proof can be found in \cite{Hajaiej2011}.
\begin{lemma}\label{fractionalgagliardonirenbergineq}
	Let $p,p_0,p_1\in(1,\infty)$ and $\kappa\in[0,s)$ with $s\in(0,\infty)$. Then, it holds for all $f\in L^{p_0}(\mb{R}^n)\cap \dot{H}^{s}_{p_1}(\mb{R}^n)$
	\begin{equation*}
	\|f\|_{\dot{H}^{\kappa}_{p}(\mb{R}^n)}\lesssim\|f\|_{L^{p_0}(\mb{R}^n)}^{1-\gamma}\,\|f\|^{\gamma}_{\dot{H}^{s}_{p_1}(\mb{R}^n)},
	\end{equation*}
	where $\gamma=\left(\frac{1}{p_0}-\frac{1}{p}+\frac{\kappa}{n}\right)\Big/\left(\frac{1}{p_0}-\frac{1}{p_1}+\frac{s}{n}\right)\in\left[\frac{\kappa}{s},1\right]$.
\end{lemma}

\subsection{Lower regular Sobolev solution}
It is well-known that the study of lower regular Sobolev solution is more challenging than the study of higher regular one since the Sobolev embedding theory does not work well. In this part, we will study global (in time) existence of small data Sobolev solutions with low regularity in the evolution space $X_s(T)$, in which we will focus on the cases $s\in[1/2,2]$ if $n=2$, and $s\in(0,2]$ if $n\geqslant 3$. 
\begin{theorem}\label{THM_GESDS_Lower_Regular}
	Let $\tau\in(0,\beta)$. Let us consider $s\in[1/2,2]$ if $n=2$, and $s\in(0,2]$ if $n\geqslant 3$. We suppose that $p\geqslant 2$, and $p\leqslant 2n/(n-s)$ if $s<n\leqslant 3s$, $p\leqslant n/(n-2s)$ if $3s<n\leqslant 4s$. If
	\begin{align}\label{Assumption_P}
	p\begin{cases}
	>5&\mbox{if}\ \ n=2, s\in[1/2,1),\\
	>s+3&\mbox{if}\ \ n=2, s\in[1,2],\\
	\geqslant (n+3)/(n-1)&\mbox{if}\ \ 3\leqslant n\leqslant 6,s\in(0,1),\\
	>(n+2)/(n-1)&\mbox{if}\ \ 3\leqslant n\leqslant 6,s\in[1,2],\\
	\geqslant\max\{3n/2-1,n+3\}/(n-1)&\mbox{if}\ \ n\geqslant 7,s\in(0,1),\\
	\geqslant (3n/2-1)/(n-1)&\mbox{if}\ \ n\geqslant 7,s\in[1,2],
	\end{cases}
	\end{align}
	there exists a sufficiently small constant $\epsilon>0$ such that for $u_2\in L^2(\mb{R}^n)\cap L^1(\mb{R}^n)$ satisfying the assumption $\|u_2\|_{L^2(\mb{R}^n)\cap L^1(\mb{R}^n)}\leqslant\epsilon$, there is a uniquely determined global (in time) Sobolev solution
	\begin{align*}
	u\in\ml{C}([0,\infty),H^s(\mb{R}^n))
	\end{align*}
	to the semilinear MGT equation \eqref{Semi_Linear_MGT_Dissipative}. Furthermore, the solution fulfills the following estimates:
	\begin{align*}
	\|u(t,\cdot)\|_{L^2(\mb{R}^n)}&\lesssim g_n(t)\|u_2\|_{L^2(\mb{R}^n)\cap L^1(\mb{R}^n)},\\
	\|\,|D|^su(t,\cdot)\|_{L^2(\mb{R}^n)}&\lesssim\tilde{g}_{n,s}(t)\|u_2\|_{L^2(\mb{R}^n)\cap L^1(\mb{R}^n)}.
	\end{align*}
\end{theorem}
\begin{exam}
	Let us consider $s=2$. Then, the observation of Theorem \ref{THM_GESDS_Lower_Regular} with $s=2$ shows the global (in time) small data Sobolev solution (in the classical energy sense)
	\begin{align*}
	u\in\ml{C}([0,\infty),H^2(\mb{R}^n))
	\end{align*} to the semilinear MGT equation \eqref{Semi_Linear_MGT_Dissipative} with $\tau\in(0,\beta)$ uniquely exists providing that
	\begin{itemize}
		\item when $n=2$, we assume $p>5$;
		\item when $n=3,4$, we assume $(n+2)/(n-1)<p\leqslant 2n/(n-2)$;
		\item when $n=5,6$, we assume $2\leqslant p\leqslant 2n/(n-2)$;
		\item when $n=7,8$, we assume  $2\leqslant p\leqslant n/(n-4)$.
	\end{itemize}
\end{exam}
\begin{remark}
	Comparing the result of the linearized problem in Theorem \ref{Thm_Lm-L2_Est} with $m=1$, the estimates stated in Theorem \ref{THM_GESDS_Lower_Regular} are no loss of decay with respect to the corresponding linear Cauchy problem.
\end{remark}
\begin{remark}
	Indeed, one may also follow the proof of Theorem \ref{THM_GESDS_Lower_Regular} to prove global (in time) existence results for other regularity assumptions on initial data. By considering $u_2\in L^2(\mb{R}^n)\cap L^m(\mb{R}^n)$ for $m\in(1,2)$, one just need to use Theorem \ref{Thm_Lm-L2_Est}, and the lower bound of the exponent $p\geqslant 2$ will be replaced by $p\geqslant 2/m$ due to the application of the fractional Gagliardo-Nirenberg inequality. 
\end{remark}
\begin{remark}
	From the restriction $n\leqslant4s$ in Theorem \ref{THM_GESDS_Lower_Regular}, we should control the dimension satisfying $n\leqslant 8$ due to $s\in(0,2]$. For the global (in time) existence result in high-dimensional space $n\geqslant 9$ with additional $L^1$ data, one may study higher regular Sobolev solution, i.e. 
	\begin{align*}
	u\in\ml{C}([0,\infty),H^s(\mb{R}^n))\ \ \mbox{with}\ \ s\in(2,\infty).
	\end{align*}
	We should emphasize that due to $s\in(2,\infty)$ in Theorem \ref{Thm_Lm-L2_Est}, it is necessary to estimate
	\begin{align*}
	\|\,|u(\sigma,\cdot)|^p\|_{\dot{H}^{s-2}(\mb{R}^n)}\ \ \mbox{and}\ \ \|\,|u(\sigma,\cdot)|^p-|v(\sigma,\cdot)|^p\|_{\dot{H}^{s-2}(\mb{R}^n)}.
	\end{align*}
	To estimate the first norm, one may apply the fractional chain rule with the additional restriction $p>\lceil s-2\rceil$. While in the estimate of the second norm, the main tools are the fractional Leibniz rule and the fractional chain rule (see \cite{Grafakos2014} and \cite{PalmieriReissig2018}) carrying a stronger condition $p>1+\lceil s-2\rceil\geqslant 2$. Furthermore, if $s-2>n/2$ and $p>s-1$, one may apply the fractional powers rule to estimate the last mentioned norms to prove existence of large regular ($s>n/2+2$) Sobolev solutions.
\end{remark}
\begin{proof}
	To begin with the proof, we construct the time-weighted norm for the evolution space $X_s(T)$ with $s\in(0,2]$ for $T>0$ by
	\begin{align*}
	\|u\|_{X_s(T)}:=\sup\limits_{t\in[0,T]}\left((g_n(t))^{-1}\|u(t,\cdot)\|_{L^2(\mb{R}^n)}+(\tilde{g}_{n,s}(t))^{-1}\|u(t,\cdot)\|_{\dot{H}^s(\mb{R}^n)}\right).
	\end{align*}
	From Theorem \ref{Thm_Lm-L2_Est} with $m=1$, we easily get
	\begin{align*}
	\|u^{\lin}\|_{X_s(T)}&\lesssim\|u_2\|_{L^2(\mb{R}^n)\cap L^1(\mb{R}^n)}.
	\end{align*}
	Thus, we may claim that $u^{\lin}\in X_s(T)$ for any $s\in(0,2]$. In the view of the desired inequality \eqref{Cruc-01}, we just have to justify the next one:
	\begin{align*}
	\|u^{\non}\|_{X_s(T)}\lesssim\|u\|_{X_s(T)}^p.
	\end{align*}
	
	Initially, we apply the derived $(L^2\cap L^1)-L^2$ estimate stated in Theorem \ref{Thm_Lm-L2_Est} in the interval $[0,t]$ to get
	\begin{align*}
	\|u^{\non}(t,\cdot)\|_{L^2(\mb{R}^n)}\lesssim\int_0^{t}g_n(t-\sigma)\|\,|u(\sigma,\cdot)|^p\|_{L^2(\mb{R}^n)\cap L^1(\mb{R}^n)}\mathrm{d}\sigma.
	\end{align*}
	To estimate the power nonlinear term in the norm, we employ the fractional Gagliardo-Nirenberg inequality that
	\begin{align*}
	\|\,|u(\sigma,\cdot)|^p\|_{ L^1(\mb{R}^n)}&=\|u(\sigma,\cdot)\|_{L^p(\mb{R}^n)}^p\lesssim(g_n(\sigma))^{(1-\gamma_1)p}(\tilde{g}_{n,s}(\sigma))^{\gamma_1p}\|u\|_{X_s(\sigma)}^p,\\
	\|\,|u(\sigma,\cdot)|^p\|_{L^2(\mb{R}^n)}&=\|u(\sigma,\cdot)\|_{L^{2p}(\mb{R}^n)}^p\lesssim (g_n(\sigma))^{(1-\gamma_2)p}(\tilde{g}_{n,s}(\sigma))^{\gamma_2p}\|u\|_{X_s(\sigma)}^p,
	\end{align*}
	where the parameters are $\gamma_1:=\frac{n}{s}\left(\frac{1}{2}-\frac{1}{p}\right)\in[0,1]$ and $\gamma_2:=\frac{n}{s}\left(\frac{1}{2}-\frac{1}{2p}\right)\in[0,1]$.\\
	The previous restrictions lead to
	\begin{align}\label{Condition_G-N}
	2\leqslant p
	\begin{cases}
	<\infty&\mbox{if}\ \ 1<n\leqslant s,\\
	\leqslant 2n/(n-s)&\mbox{if}\ \ s<n\leqslant 3s,\\
	\leqslant n/(n-2s)&\mbox{if}\ \ 3s<n\leqslant 4s.
	\end{cases}
	\end{align}
	Here, the restriction for $n\leqslant 4s$ comes from the nonempty set of $p\in[2,n/(n-2s)]$.\\
	Obviously, we know that
	\begin{align*}
	1>\frac{\tilde{g}_{n,s}(\sigma)}{g_n(\sigma)}=\begin{cases}
	(1+\sigma)^{-\frac{s}{2}}(\ln(\mathrm{e}+\sigma))^{-\frac{1}{2}}&\mbox{if}\ \ n=2,s\in[1/2,2],\\
	(1+\sigma)^{-\frac{s}{2}}&\mbox{if}\ \  n\geqslant 3,s\in(0,2],
	\end{cases}
	\end{align*}
	which implies from $\gamma_1<\gamma_2$,
	\begin{align*}
	\|\,|u(\sigma,\cdot)|^p\|_{L^2(\mb{R}^n)\cap  L^1(\mb{R}^n)}\lesssim
	(g_n(\sigma))^{(1-\gamma_1)p}(\tilde{g}_{n,s}(\sigma))^{\gamma_1p}\|u\|_{X_s(\sigma)}^p.
	\end{align*}
	For one thing, by some direct computations, they yield
	\begin{align*}
	(g_n(\sigma))^{(1-\gamma_1)p}(\tilde{g}_{n,s}(\sigma))^{\gamma_1p}=\begin{cases}
	(1+\sigma)^{-\frac{p}{2}+1}(\ln(\mathrm{e}+\sigma))^{\frac{(s-1)p+2}{2s}}&\mbox{if}\ \ n=2,s\in[1/2,2],\\
	(1+\sigma)^{-\frac{(n-1)p}{2}+\frac{n}{2}}&\mbox{if}\ \ n\geqslant 3,s\in(0,2],
	\end{cases}
	\end{align*}
	and
	\begin{align*}
	(g_n(\sigma))^{(1-\gamma_2)p}(\tilde{g}_{n,s}(\sigma))^{\gamma_2p}=\begin{cases}
	(1+\sigma)^{-\frac{p}{2}+\frac{1}{2}}(\ln(\mathrm{e}+\sigma))^{\frac{(s-1)p+1}{2s}}&\mbox{if}\ \ n=2,s\in[1/2,2],\\
	(1+\sigma)^{-\frac{(n-1)p}{2}+\frac{n}{4}}&\mbox{if}\ \ n\geqslant 3,s\in(0,2].
	\end{cases}
	\end{align*}
	According to the assumption
	\begin{align}\label{Condition_001}
	p\begin{cases}
	>\max\{s+3,4\}&\mbox{if}\ \ n=2,s\in[1/2,2],\\
	>(n+2)/(n-1)&\mbox{if} \ \ 3\leqslant n\leqslant 6,s\in(0,2],\\
	\geqslant(3n-2)/(2n-2)&\mbox{if} \ \ n\geqslant7,s\in(0,2],
	\end{cases}
	\end{align}
	it is true that
	\begin{align*}
	\int_0^{t/2}(g_n(\sigma))^{(1-\gamma_1)p}(\tilde{g}_{n,s}(\sigma))^{\gamma_1p}\mathrm{d}\sigma\lesssim 1,
	\end{align*}
	and
	\begin{align*}
	(g_n(t))^{(1-\gamma_1)p}(\tilde{g}_{n,s}(t))^{\gamma_1p}\int_{t/2}^tg_n(t-\sigma)\mathrm{d}\sigma&\lesssim\begin{cases}
	(1+t)^{-\frac{p}{2}+\frac{3}{2}}(\ln(\mathrm{e}+t))^{\frac{(s-1)p+s+2}{2s}}&\mbox{if}\ \ n=2,\\
	(1+t)^{-\frac{(n-1)p}{2}+\frac{3}{2}+\frac{n}{4}}&\mbox{if}\ \ 3\leqslant n\leqslant 5,\\
	(1+t)^{-\frac{5p}{2}+3}\ln(\mathrm{e}+t)&\mbox{if}\ \ n=6,\\
	(1+t)^{-\frac{(n-1)p}{2}+\frac{n}{2}}&\mbox{if}\ \ n\geqslant 7,
	\end{cases}\\
	&\lesssim g_n(t).
	\end{align*}
	Then, by dividing $[0,t]$ into $[0,t/2]$ and $[t/2,t]$, one may immediately arrive at
	\begin{align*}
	\|u^{\non}(t,\cdot)\|_{L^2(\mb{R}^n)}&\lesssim g_n(t)\|u\|_{X_s(T)}^p\int_0^{t/2}(g_n(\sigma))^{(1-\gamma_1)p}(\tilde{g}_{n,s}(\sigma))^{\gamma_1p}\mathrm{d}\sigma\\
	&\quad+(g_n(t))^{(1-\gamma_1)p}(\tilde{g}_{n,s}(t))^{\gamma_1p}\|u\|_{X_s(T)}^p\int_{t/2}^tg_n(t-\sigma)\mathrm{d}\sigma\\
	&\lesssim g_n(t)\|u\|_{X_s(T)}^p,
	\end{align*}
	where we used $\|u\|_{X_s(\sigma)}\lesssim \|u\|_{X_s(T)}$ for any $\sigma\in[0,T]$ and taking account into the fact that
	\begin{align*}
	(1+t-\sigma)\approx(1+t)\ \ \mbox{for}\ \ \sigma\in[0,t/2]\ \ \mbox{and}\ \ (1+\sigma)\approx (1+t)\ \ \mbox{for}\ \ \sigma\in[t/2,t].
	\end{align*}
	
	Next, we will estimate the solution in the $\dot{H}^s$ norm. At this time, we employ the obtained $(L^2\cap L^1)-L^2$ estimate in $[0,t/2]$, and $L^2-L^2$ estimate in $[t/2,t]$ leading to
	\begin{align*}
	\|u^{\non}(t,\cdot)\|_{\dot{H}^s(\mb{R}^n)}&\lesssim\int_0^{t/2}\tilde{g}_{n,s}(t-\sigma)\|\,|u(\sigma,\cdot)|^p\|_{L^2(\mb{R}^n)\cap L^1(\mb{R}^n)}\mathrm{d}\sigma+\int_{t/2}^th_s(t-\sigma)\|\,|u(\sigma,\cdot)|^p\|_{L^2(\mb{R}^n)}\mathrm{d}\sigma\\
	&\lesssim \tilde{g}_{n,s}(t)\|u\|_{X_s(T)}^p\int_0^{t/2}(g_n(\sigma))^{(1-\gamma_1)p}(\tilde{g}_{n,s}(\sigma))^{\gamma_1p}\mathrm{d}\sigma\\
	&\quad+(g_n(t))^{(1-\gamma_2)p}(\tilde{g}_{n,s}(t))^{\gamma_2p}(1+t)h_s(t)\|u\|_{X_s(T)}^p\\
	&\lesssim \tilde{g}_{n,s}(t)\|u\|_{X_s(T)}^p,
	\end{align*}
	where we used our assumption \eqref{Condition_001} and additionally,
	\begin{align}\label{Condition_111}
	p\begin{cases}
	> 5&\mbox{if} \ \ n=2, s\in[1/2,1)\\
	> 4&\mbox{if}\ \ n=2,s\in[1,2],\\
	\geqslant(n+3)/(n-1)&\mbox{if} \ \ n\geqslant 3,s\in(0,1),\\
	\geqslant(n+2)/(n-1)&\mbox{if} \ \ n\geqslant 3,s\in[1,2],
	\end{cases}
	\end{align}
	to derive
	\begin{align*}
	1&\gtrsim(g_n(t))^{(1-\gamma_2)p}(\tilde{g}_{n,s}(t))^{\gamma_2p}(1+t)h_s(t)(\tilde{g}_{n,s}(t))^{-1}\\
	&=\begin{cases}
	(1+t)^{-\frac{p}{2}+\frac{5}{2}}(\ln(\mathrm{e}+t))^{\frac{(s-1)p+1}{2s}}&\mbox{if}\ \ n=2,s\in[1/2,1),\\
	(1+t)^{-\frac{p}{2}+2}(\ln(\mathrm{e}+t))^{\frac{(s-1)p+1}{2s}}&\mbox{if}\ \ n=2,s\in[1,2],\\
	(1+t)^{-\frac{(n-1)p}{2}+\frac{n}{2}+\frac{3}{2}}&\mbox{if}\ \ n\geqslant 3,s\in(0,1),\\
	(1+t)^{-\frac{(n-1)p}{2}+\frac{n}{2}+1}&\mbox{if}\ \ n\geqslant 3,s\in[1,2].
	\end{cases}
	\end{align*}
	By assuming \eqref{Condition_G-N}, \eqref{Condition_001}, \eqref{Condition_111} and summarizing the derived estimates, it is proved that the operator $N$ maps $X_s(T)$ into itself, namely, $Nu\in X_s(T)$.
	
	Finally, with the aim of proving the Lipschitz condition, we may take two solutions $u,v\in X_s(T)$. From the derived result of \eqref{Cruc-01}, it is clear that $Nu,Nv\in X_s(T)$. Therefore, we have
	\begin{align*}
	\|Nu-Nv\|_{X_s(T)}=\left\|\int_0^tK(t-\sigma,x)\ast_{(x)}(|u(\sigma,x)|^p-|v(\sigma,x)|^p)\mathrm{d}\sigma\right\|_{X_s(T)}.
	\end{align*}
	We assume that \eqref{Condition_G-N} and \eqref{Assumption_P} hold. For the estimate in the $L^2$ norm, we apply H\"older's inequality and the fractional Gagliardo-Nirenberg inequality to arrive at
	\begin{align*}
	&\|(Nu-Nv)(t,\cdot)\|_{L^2(\mb{R}^n)}\\
	&\lesssim\int_0^tg_n(t-\sigma)\|\,|u(\sigma,\cdot)|^p-|v(\sigma,\cdot)|^p\|_{L^2(\mb{R}^n)\cap L^1(\mb{R}^n)}\mathrm{d}\sigma\\
	&\lesssim\int_0^tg_n(t-\sigma)\|u(\sigma,\cdot)-v(\sigma,\cdot)\|_{L^p(\mb{R}^n)}\left(\|u(\sigma,\cdot)\|_{L^p(\mb{R}^n)}^{p-1}+\|v(\sigma,\cdot)\|_{L^p(\mb{R}^n)}^{p-1}\right)\mathrm{d}\sigma\\
	&\quad+\int_0^tg_n(t-\sigma)\|u(\sigma,\cdot)-v(\sigma,\cdot)\|_{L^{2p}(\mb{R}^n)}\left(\|u(\sigma,\cdot)\|_{L^{2p}(\mb{R}^n)}^{p-1}+\|v(\sigma,\cdot)\|_{L^{2p}(\mb{R}^n)}^{p-1}\right)\mathrm{d}\sigma\\
	&\lesssim\int_0^tg_n(t-\sigma)(g_n(\sigma))^{(1-\gamma_1)p}(\tilde{g}_{n,s}(\sigma))^{\gamma_1p}\mathrm{d}\sigma\,\|u-v\|_{X_s(T)}\left(\|u\|_{X_s(T)}^{p-1}+\|v\|_{X_s(T)}^{p-1}\right)\\
	& \lesssim g_n(t) \|u-v\|_{X_s(T)}\left(\|u\|_{X_s(T)}^{p-1}+\|v\|_{X_s(T)}^{p-1}\right).
	\end{align*}
	By repeating the same approach as before, we conclude
	\begin{align*}
	\|(Nu-Nv)(t,\cdot)\|_{\dot{H}^s(\mb{R}^n)}\lesssim \tilde{g}_{n,s}(t) \|u-v\|_{X_s(T)}\left(\|u\|_{X_s(T)}^{p-1}+\|v\|_{X_s(T)}^{p-1}\right).
	\end{align*}
	
	Therefore, the crucial estimates \eqref{Cruc-01} and \eqref{Cruc-02} are valid. By using Banach's fixed point theorem, there exists a unique determined global (in time) low regular Sobolev solution to the semilinear MGT equation \eqref{Semi_Linear_MGT_Dissipative}. The proof is complete.
\end{proof}

\section{Nonexistence of global (in time) weak solutions}\label{Section_Blow-up}
Before showing our main result on blow-up of solutions, let us first give a definition of weak solutions to the semilinear MGT equation \eqref{Semi_Linear_MGT_Dissipative}.
\begin{defn}\label{Defn_Weak_Solution}
	Let $p>1$. We say $u\in L_{\mathrm{loc}}^p([0,\infty)\times\mb{R}^n)$ is a global (in time) weak solution to the semilinear MGT equation \eqref{Semi_Linear_MGT_Dissipative} if the integral equality 
	\begin{align}\label{Integral Equality}
	&\int_0^{\infty}\int_{\mb{R}^n}u(t,x)(-\tau\psi_{ttt}(t,x)+\psi_{tt}(t,x)-\Delta\psi(t,x)+\beta\Delta\psi_t(t,x))\mathrm{d}x\mathrm{d}t\notag\\
	&=\tau\int_{\mb{R}^n}u_2(x)\psi(0,x)\mathrm{d}x+\int_0^{\infty}\int_{\mb{R}^n}|u(t,x)|^p\psi(t,x)\mathrm{d}x\mathrm{d}t
	\end{align}
	holds for any $\psi\in\ml{C}_0^{\infty}([0,\infty)\times\mb{R}^n)$.
\end{defn}
\begin{theorem}
	Let $\tau\in(0,\beta)$. Let us assume that $u_2\in L^1(\mb{R}^n)$ and 
	\begin{align*}
	\int_{\mb{R}^n}u_2(x)\mathrm{d}x>0.
	\end{align*}
	Then, there exist no any the global (in time) weak solutions to the semilinear MGT equation \eqref{Semi_Linear_MGT_Dissipative} in the sense of  Definition \ref{Defn_Weak_Solution} providing that the exponent of nonlinearity satisfies
	\begin{align*}
	1<p\begin{cases}
	<\infty&\mbox{if}\ \ n=1,\\
	\leqslant (n+1)/(n-1)&\mbox{if}\ \ n\geqslant 2.
	\end{cases}
	\end{align*}
\end{theorem}
\begin{remark}
	In the one-dimensional case, every weak solution according to Definition \ref{Defn_Weak_Solution} blows up for any $1<p<\infty$, which means that the result in 1D is optimal. 
\end{remark}
\begin{remark}
	We may derive blow-up results for other regularity assumptions on initial data. Let us assume $u_2\in L^m(\mb{R}^n)$ with $m\in(1,2)$ and
	\begin{align*}
	u_2(x)\gtrsim|x|^{-\frac{n}{m}}(\ln(1+|x|))^{-1}\ \ \mbox{for} \ \ |x|\gg 1.
	\end{align*}
	Then, one may also prove blow-up of weak solutions to the semilinear MGT equation \eqref{Semi_Linear_MGT_Dissipative} providing that $1<p<\infty$ if $n=1$, and $1<p<(n+m)/(n-m)$ if $n\geqslant2$. The proof is strictly following those of Theorem 4.1 in \cite{Chen-Dao2020}. 
\end{remark}
\begin{proof}
	Let us now introduce two bump functions $\eta\in\ml{C}_0^{\infty}([0,\infty))$ and $\phi\in\ml{C}_0^{\infty}(\mb{R}^n)$ such that $\eta=\eta(t)$ is decreasing with $\eta=1$ on $[0,1/2]$ and $\mathrm{supp}\,\eta\subset[0,1]$; $\phi=\phi(x)$ is radial symmetric, decreasing with respect to $|x|$ with $\phi=1$ on $B_{1/2}$ and $\mathrm{supp}\,\phi\subset B_1$. Moreover, we assume
	\begin{align}\label{assumption eta phi 1}
	(\eta(t))^{-\frac{p'}{p}}\left(|\eta'''(t)|^{p'}+|\eta''(t)|^{p'}+|\eta'(t)|^{p'}\right)&\leqslant C,\\
	(\phi(x))^{-\frac{p'}{p}}|\Delta\phi(x)|^{p'}&\leqslant C,\label{assumption eta phi 2}
	\end{align}
	where $p'$ is the conjugate of $p$, i.e. $1/p+1/p'=1$, and $C$ is a positive constant, with $\eta,\phi\in[0,1]$.
	
	To begin with, we define a test function
	\begin{align*}
	\psi_R(t,x):=\eta_R(t)\phi_R(x):=\eta(t/R)\phi(x/R),
	\end{align*}
	where $R\in[1,\infty)$ is a large parameter. Furthermore, we may introduce
	\begin{align*}
	I_R:=\int_0^{\infty}\int_{\mb{R}^n}|u(t,x)|^p\psi_R(t,x)\mathrm{d}x\mathrm{d}t.
	\end{align*}
	By considering \eqref{Integral Equality} in the definition of weak solution with the test function $\psi(t,x)=\psi_R(t,x)$, one immediately has
	\begin{align*}
	&I_R+\tau\int_{\mb{R}^n} u_2(x)\phi_R(x)\mathrm{d}x\\
	&=\int_0^{\infty}\int_{\mb{R}^n}u(t,x)\left(-\tau\partial_t^3\psi_R(t,x)+\partial_t^2\psi_R(t,x)-\Delta\psi_R(t,x)+\beta\partial_t\Delta\psi_R(t,x)\right)\mathrm{d}x\mathrm{d}t\\
	&\leqslant\frac{1}{p}I_R+\frac{1}{p'}\int_0^{\infty}\int_{\mb{R}^n}(\eta_R(t)\phi_R(x))^{-\frac{p'}{p}}\left(\tau^{p'}|\mathrm{d}_t^3\eta_R(t)\phi_R(x)|^{p'}+|\mathrm{d}_t^2\eta_R(t)\phi_R(x)|^{p'}\right)\mathrm{d}x\mathrm{d}t\\
	&\quad+\frac{1}{p'}\int_0^{\infty}\int_{\mb{R}^n}(\eta_R(t)\phi_R(x))^{-\frac{p'}{p}}\left(|\eta_R(t)\Delta\phi_R(x)|^{p'}+\beta^{p'}|\mathrm{d}_t\eta_R(t)\Delta\phi_R(x)|^{p'}\right)\mathrm{d}x\mathrm{d}t,
	\end{align*}
	where we employed Young's inequality $ab\leqslant a^p/p+b^{p'}/p'$.
	
	Due to the fact that
	\begin{align*}
	\Delta\phi_R(x)=R^{-2}\Delta\phi(x/R),\ \ \mbox{and} \ \ \mathrm{d}_t^k\eta_R(t)=R^{-k}\mathrm{d}_t^k\eta(t/R)\ \  \mbox{for}\ \  k=1,2,3,
	\end{align*}
	we are able to deduce
	\begin{align*}
	I_R&\lesssim\frac{1}{p'}I_R+\tau\int_{\mb{R}^n}u_2(x)\phi_R(x)\mathrm{d}x\\
	&\lesssim R^{-2p'}\int_0^{\infty}\int_{\mb{R}^n}\left((\eta(t/R))^{-\frac{p'}{p}}+|\eta''(t/R)|^{p'}\phi(x/R)+\eta(t/R)(\phi(x/R))^{-\frac{p'}{p}}|\Delta\phi(x/R)|^{p'}\right)\mathrm{d}x\mathrm{d}t\\
	&\quad+R^{-3p'}\int_0^{\infty}\int_{\mb{R}^n}(\eta(t/R))^{-\frac{p'}{p}}\left(|\eta'''(t/R)|^{p'}\phi(x/R)+|\eta'(t/R)|^{p'}(\phi(x/R))^{-\frac{p'}{p}}|\Delta\phi(x/R)|^{p'}\right)\mathrm{d}x\mathrm{d}t\\
	&\lesssim R^{-2p'+1+n}+R^{-3p'+1+n}\lesssim R^{-2p'+1+n},
	\end{align*}
	where the conditions for test functions in \eqref{assumption eta phi 1} and \eqref{assumption eta phi 2} were used. Moreover, we applied our assumption on initial data such that
	\begin{align*}
	\tau\int_{\mb{R}^n}u_2(x)\mathrm{d}x>0\ \ \Rightarrow\ \ \tau\int_{\mb{R}^n}u_2(x)\phi_R(x)\mathrm{d}x> 0
	\end{align*}
	for any $R\gg 1$ because of the fact that
	\begin{align*}
	\lim\limits_{R\to\infty}\int_{\mb{R}^n}u_2(x)\phi_R(x)\mathrm{d}x=\int_{\mb{R}^n}u_2(x)\mathrm{d}x.
	\end{align*}

	According to the condition on $p$ leading to $-2p'+1+n<0$, and letting $R\to\infty$, we get $\lim_{R\to\infty}I_R=0$, which leads to $u=0$ a.e., however, this contradicts to our assumption. In other words, the global (in time) weak solution does not exist.
	
	To prove the blow-up result in the limit case when $p=(n+1)/(n-1)$ if $n\geqslant2$, we can also conclude the contradiction that $\lim_{R\to\infty}I_R=0$ by following the approach in \cite{Zhang01}, i.e. the monotone convergence theorem and the dominant convergence theorem. All in all, the proof is completed.
\end{proof}

\section*{Acknowledgments}
The second author was supported in part by Grant-in-Aid for scientific Research (C) 20K03682 of JSPS. The authors thank  Michael Reissig (TU Bergakademie Freiberg) for the suggestions in the preparation of the paper.

\end{document}